\documentclass[11pt,a4paper,reqno]{amsart}
\usepackage[english]{babel}
\usepackage[applemac]{inputenc}
\usepackage[T1]{fontenc}
\usepackage{palatino}
\usepackage{amsmath}
\usepackage{amssymb}
\usepackage{amsthm}
\usepackage{amsfonts}
\usepackage{graphicx}
\usepackage{mathtools}

\usepackage[colorlinks = true, citecolor = black]{hyperref}
\author{Tuomas Orponen}
\title{Absolute continuity and $\alpha$-numbers on the real line}
\address{University of Helsinki, Department of Mathematics and Statistics}
\thanks{The research was supported by the grants 274512 and 309365 of the Finnish Academy.}
\email{tuomas.orponen@helsinki.fi}

\newcommand{\R}{\mathbb{R}}

\newcommand{\N}{\mathbb{N}}

\newcommand{\W}{\mathbb{W}}
\newcommand{\Z}{\mathbb{Z}}

\newcommand{\calT}{\mathcal{T}}
\newcommand{\calL}{\mathcal{L}}
\newcommand{\calD}{\mathcal{D}}

\newcommand{\calB}{\mathcal{B}}

\newcommand{\calS}{\mathcal{S}}

\newcommand{\spt}{\operatorname{spt}}

\newcommand{\card}{\operatorname{card}}
\newcommand{\dist}{\operatorname{dist}}

\newcommand{\Leaves}{\mathbf{Leaves}}
\newcommand{\Top}{\mathbf{Top}}

\newcommand{\ch}{\mathbf{ch}}

\newcommand{\calG}{\mathcal{G}}

\newcommand{\Tail}{\mathbf{Tail}}
\newcommand{\Tip}{\mathbf{Tip}}

\numberwithin{equation}{section}

\theoremstyle{plain}
\newtheorem{thm}[equation]{Theorem}

\newtheorem{lemma}[equation]{Lemma}

\newtheorem{ex}[equation]{Example}
\newtheorem{cor}[equation]{Corollary}
\newtheorem{proposition}[equation]{Proposition}
\newtheorem{question}{Question}

\newtheorem{obs}{Observation}

\theoremstyle{definition}

\newtheorem{definition}[equation]{Definition}

\theoremstyle{remark}
\newtheorem{remark}[equation]{Remark}

\begin{document}
\begin{abstract} Let $\mu,\nu$ be Radon measures on $\R$, with $\mu$ non-atomic and $\nu$ doubling, and write $\mu = \mu_{a} + \mu_{s}$ for the Lebesgue decomposition of $\mu$ relative to $\nu$. For an interval $I \subset \R$, define $\alpha_{\mu,\nu}(I) := \W_{1}(\mu_{I},\nu_{I})$, the Wasserstein distance of normalised blow-ups of $\mu$ and $\nu$ restricted to $I$. Let $\calS_{\nu}$ be the square function
\begin{displaymath} \calS^{2}_{\nu}(\mu) = \sum_{I \in \calD} \alpha_{\mu,\nu}^{2}(I)\chi_{I}, \end{displaymath}
where $\calD$ is the family of dyadic intervals of side-length at most one. I prove that $\calS_{\nu}(\mu)$ is finite $\mu_{a}$ almost everywhere, and infinite $\mu_{s}$ almost everywhere. I also prove a version of the result for a non-dyadic variant of the square function $\calS_{\nu}(\mu)$. The results answer the simplest "$n = d = 1"$ case of a problem of J. Azzam, G. David and T. Toro. 
\end{abstract}

\maketitle

\tableofcontents

\section{Introduction}

\subsection{Wasserstein distance and $\alpha$-numbers} In this paper, $\mu$ and $\nu$ are non-zero Radon measures on $\R$. The measure $\nu$ is generally assumed to be either \emph{dyadically doubling} or \emph{globally doubling}. Dyadically doubling means that
\begin{equation}\label{doub} \nu(\hat{I}) \leq C\nu(I), \quad I \in \calD, \end{equation}
where $\calD$ is the standard family of dyadic intervals, and $\hat{I}$ is the \emph{parent} of $I$, that is, the smallest interval in $\calD$ strictly containing $I$. Globally doubling means that $\nu(B(x,2r)) \leq C\nu(B(x,r))$ for $x \in \R$ and $r > 0$; in particular, this implies $\spt \nu = \R$. The main example for $\nu$ is the Lebesgue measure $\calL$, and the proofs in this particular case would differ little from the ones presented below. No \emph{a priori} homogeneity is assumed of $\mu$.

\begin{definition}[Wasserstein distance] I will use the following definition of the (first) Wasserstein distance: given two Radon measures measures $\nu_{1},\nu_{2}$ on $[0,1]$, set
\begin{displaymath} \mathbb{W}_{1}(\nu_{1},\nu_{2}) := \sup_{\psi} \left| \int \psi \, d\nu_{1} - \int \psi \, d\nu_{2} \right|, \end{displaymath}
where the $\sup$ is taken over all $1$-Lipschitz functions $\psi \colon \R \to \R$, which are supported on $[0,1]$. Such functions will be called \emph{test functions}. A slightly different -- and also quite common -- definition would allow the $\sup$ to run over all $1$-Lipschitz functions $\psi \colon [0,1] \to \R$. To illustrate the difference, let $\nu_{1} = \delta_{0}$ and $\nu_{2} = \delta_{1}$. Then $\W_{1}(\nu_{1},\nu_{2}) = 0$, but the alternative definition, say $\tilde{\W}_{1}$, would give $\tilde{\W}_{1}(\nu_{1},\nu_{2}) = 1$. The main reason for using $\W_{1}$ instead of $\tilde{\W}_{1}$ in this paper is to comply with the definitions in \cite{ADT,ADT2}.

\end{definition}

As in the paper \cite{ADT} of J. Azzam, G. David and T. Toro, I make the following definition:
\begin{definition}[$\alpha$-numbers]\label{alphas} Assume that $I \subset \R$ is an interval. Define
\begin{displaymath} \alpha_{\mu,\nu}(I) := \W_{1}(\mu_{I},\nu_{I}), \end{displaymath}
where $\mu_{I}$ and $\nu_{I}$ are normalised blow-ups of $\mu$ and $\nu$ restricted to $I$. More precisely, let $T_{I}$ be the increasing affine mapping taking $\overline{I}$ to $[0,1]$, and define
\begin{displaymath} \mu_{I} := \frac{T_{I\sharp}(\mu|_{I})}{\mu(I)} \quad \text{and} \quad \nu_{I} := \frac{T_{I\sharp}(\nu|_{I})}{\nu(I)}. \end{displaymath}
If $\mu(I) = 0$ (or $\nu(I) = 0$), define $\mu_{I} \equiv 0$ (or $\nu_{I} \equiv 0$).
\end{definition}

The quantity defined above is somewhat awkward to work with, as it lacks (see Example \ref{counterEx}) the following desirable stability property: if $I,J \subset \R$ are intervals of comparable length, and $I \subset J$, then $\alpha_{\mu,\nu}(I) \lesssim \alpha_{\mu,\nu}(J)$. Chiefly for this reason, I also need to consider the following "smooth" $\alpha$-numbers; the definition below is essentially the same as the one given by Azzam, David and Toro in \cite[Section 5]{ADT2}:
\begin{definition}[Smooth $\alpha$-numbers]\label{smoothAlphas} Let $\varphi := \dist(\cdot,\R \setminus (0,1))$. For an interval $I \subset \R$, define $\alpha_{s,\mu,\nu}(I) := \W_{1}(\mu_{\varphi,I},\nu_{\varphi,I})$, where 
\begin{displaymath} \mu_{\varphi,I} := \frac{T_{I\sharp}(\mu|_{I})}{\mu(\varphi_{I})} \quad \text{and} \quad  \nu_{\varphi,I} := \frac{T_{I\sharp}(\nu|_{I})}{\nu(\varphi_{I})}. \end{displaymath}
Here $T_{I}$ is the map from Definition \ref{alphas}, $\varphi_{I} = \varphi \circ T_{I}$, and $\mu(\varphi_{I}) = \int \varphi_{I} \, d\mu$. If $\mu(\varphi_{I}) = 0$ (or $\nu(\varphi_{I}) = 0$), set $\mu_{\varphi,I} \equiv 0$ (or $\nu_{\varphi,I} \equiv 0$). \end{definition}

The only difference between the numbers $\alpha_{\mu,\nu}(I)$ and $\alpha_{s,\mu,\nu}(I)$ is in the normalisation of the measures $\mu_{I},\varphi_{I}$ and $\mu_{\varphi,I},\nu_{\varphi,I}$: if $I$ is closed, the measures $\mu_{I},\nu_{I}$ are probability measures on $[0,1]$, while $\mu_{\varphi,I}([0,1]) = \mu(I)/\mu(\varphi_{I})$. The numbers $\alpha_{s,\mu,\nu}(I)$ enjoy the stability property alluded to above. Moreover, if either $\mu$ or $\nu$ is a doubling, one has $\alpha_{s,\mu,\nu}(I) \lesssim \alpha_{\mu,\nu}(I)$. These facts are contained in Proposition \ref{basicProperties} (or see \cite[Section 5]{ADT2}).

\begin{remark}\label{originalAlphas} The $\alpha$-numbers were first introduced by X. Tolsa in \cite{To2}, where he used them to characterise the uniform rectifiability of Ahlfors-David regular measures in $\R^{d}$. Tolsa's original definition of the $\alpha$-numbers has a different, asymmetric, normalisation compared to either $\alpha_{\mu,\nu}$ or $\alpha_{s,\mu,\nu}$ above, see \cite[p. 394]{To2}.
\end{remark}

\subsection{Main results} Before explaining the results in Azzam, David and Toro's paper \cite{ADT}, and their connection to the current manuscript, I emphasise that \cite{ADT} treats "$n$-dimensional" measures in $\R^{d}$, for any $1 \leq n \leq d$. For the current paper, only the case $n = d = 1$ is relevant. So, to avoid digressing too much, I need to state the results of \cite{ADT} in far smaller generality than they deserve. 

With this proviso in mind, the main results of \cite{ADT} imply the following. if $\mu$ is a doubling measure on $\R$, and the numbers $\alpha_{\mu,\calL}$ satisfy a Carleson condition of the form
\begin{equation}\label{carlesonADT} \int_{B(x,2r)} \int_{0}^{2r} \alpha_{\mu,\calL}(B(y,t)) \, \frac{dt \, d\mu y}{t} \leq C\mu(B(x,r)), \end{equation}
then $\mu$, or at least a large part of $\mu$, is absolutely continuous with respect to $\calL$, with quantitative upper and lower bounds on the density. As the authors of \cite{ADT} point out, the main shortcoming of their result is that condition \eqref{carlesonADT} imposes a hypothesis on the first powers of the numbers $\alpha_{\mu,\calL}$, whereas evidence suggests (see \cite[Remark 1.6.1]{ADT}, the discussion after \cite[Theorem 1.7]{ADT}, and \cite[Example 4.6]{ADT}) that the correct power should be two. More support for this belief comes from the following "converse" result of Tolsa \cite[Lemma 2.2]{To3}: if $\mu$ is a finite Borel measure on $\R$ then 
\begin{equation}\label{finiteSF} \int_{0}^{\infty} \tilde{\alpha}_{\mu,\calL}^{2}(x,r) \frac{dr}{r} < \infty \text{ for } \calL \text{ a.e. } x \in \R. \end{equation}
In particular, if $\mu \ll \calL$, then \eqref{finiteSF} holds for $\mu$ almost every $x \in \R$. I should again mention that this is only the easiest $n = d = 1$ case of Tolsa's result. Here $\tilde{\alpha}_{\mu,\calL}$ is a variant of the $\alpha$-number (in fact the one discussed in Remark \ref{originalAlphas}).

The purpose of this paper is to address the problem of Azzam, David and Toro in the simplest case $n = d = 1$. I show that control for the second powers of the $\alpha_{\mu,\calL}$-numbers does guarantee absolute continuity with respect to Lebesgue measure. In fact, the doubling assumption on $\mu$ can be dropped, the Carleson condition \eqref{carlesonADT} can be relaxed considerably, and the results remain valid, if $\calL$ is replaced by any doubling measure $\nu$. The results below also contain the "converse" statement, analogous to \eqref{finiteSF}. 

I prove two variants of the main result: one dyadic, and the other non-dyadic. Here is the dyadic version:

\begin{thm}\label{main} Let $\calD$ be the family of dyadic subintervals of $[0,1)$, and let $\mu,\nu$ be Borel probability measures on $[0,1)$. Assume that $\mu$ does not charge the boundaries of intervals in $\calD$, and $\nu$ is dyadically doubling. Write $\mu = \mu_{a} + \mu_{s}$ for the Lebesgue decomposition of $\mu$ relative to $\nu$, where $\mu_{a} \ll \nu$ and $\mu_{s} \perp \nu$. Finally, let $\calS_{\calD,\nu}(\mu)$ be the square function
\begin{displaymath} \calS^{2}_{\calD,\nu}(\mu) = \sum_{I \in \calD} \alpha_{\mu,\nu}^{2}(I)\chi_{I}. \end{displaymath}
Then:
\begin{itemize}
\item[(a)] $\calS_{\nu}(\mu)$ is finite $\mu_{a}$ almost surely, and
\item[(b)] $\calS_{\nu}(\mu)$ is infinite $\mu_{s}$ almost surely.
\end{itemize}
\end{thm}

In particular, 
\begin{displaymath} \sum_{I \in \calD} \alpha_{\mu,\nu}^{2}(I)\mu(I) < \infty \quad \Longrightarrow \quad \mu \ll \nu. \end{displaymath}
Heuristically, this corresponds to assuming \eqref{carlesonADT} at the scale $r = 1$, but I could not found a way to \textbf{reduce} the continuous problem to the dyadic one; on the other hand, a reduction in the other direction does not appear straightforward, either, so perhaps one needs to treat the cases separately. A caveat of the dyadic set-up is the "non-atomicity" hypothesis on $\mu$. It cannot be dispensed with: for instance, if $\mu = \delta_{x}$ for any $x \in [0,1)$, which only belongs to the interiors of finitely many dyadic intervals, then $\calS_{\calD,\calL}(\mu)$ is uniformly bounded (for instance $\calS_{\calD,\calL}(\delta_{0}) \equiv 0$), but $\mu \perp \calL$. 

Here is the non-dyadic version of the main theorem:
\begin{thm}\label{mainCont} Assume that $\mu,\nu$ are Radon measures, and $\nu$ is globally doubling. Write $\mu = \mu_{a} + \mu_{s}$, as in Theorem \ref{main}. Let $\calS_{\nu}$ be the square function
\begin{displaymath} \calS^{2}_{\nu}(\mu)(x) = \int_{0}^{1} \alpha_{s,\mu,\nu}^{2}(B(x,r)) \, \frac{dr}{r}, \qquad x \in \R, \end{displaymath}
defined via the smooth $\alpha$-numbers $\alpha_{s,\mu,\nu}$. Then,
\begin{itemize}
\item[(a)] $\calS_{\nu}(\mu)$ is finite $\mu_{a}$ almost surely, and
\item[(b)] $\calS_{\nu}(\mu)$ is infinite $\mu_{s}$ almost surely.
\end{itemize}
\end{thm}
Recall that $\alpha_{s,\mu,\nu}(B(x,r)) \lesssim \alpha_{\mu,\nu}(B(x,r))$ whenever $\nu$ is doubling, such as $\nu = \calL$, see Proposition \ref{basicProperties}. So, Theorem \ref{mainCont} has the following corollary:
\begin{cor} If $\mu$ is a Radon measure on $\R$ such that
\begin{equation}\label{form38} \int_{0}^{1} \alpha^{2}_{\mu,\calL}(B(x,t)) \, \frac{dt}{t} < \infty \end{equation}
for $\mu$ almost every $x \in \R$, then $\mu \ll \nu$. \end{cor}
The following question remains open:
\begin{question}\label{Q1} In the setting of Theorem \ref{mainCont}, is the square function in \eqref{form38} (with $\calL$ replaced by $\nu$) finite $\mu_{a}$ almost everywhere?
\end{question}
The difficulties arise from the non-stability of the numbers $\alpha_{\mu,\nu}$. See \cite[Section 5]{ADT2}, and in particular \cite[Lemma 5.3]{ADT2}, for related discussion.

Assuming the full Carleson condition \eqref{carlesonADT}, and that $\mu$ is globally doubling, the authors of \cite{ADT} prove something more quantitative than $\mu \ll \calL$; see in particular \cite[Theorem 1.9]{ADT}. The same ought to be true for the second powers of the $\alpha$-numbers, and indeed the following result can be easily deduced with the method of the current paper:
\begin{thm}\label{mainCarleson} Assume that $\mu,\nu$ are Borel probability measures on $[0,1)$, both dyadically doubling, and assume that the Carleson condition
\begin{equation}\label{carleson} \sum_{I \subset J} \alpha_{\mu,\nu}^{2}(I)\mu(I) \leq C\mu(J), \qquad J \in \calD, \end{equation}
holds for some $C \geq 1$. Then $\mu$ belongs to $A_{\infty}^{\calD}(\nu)$, the dyadic $A_{\infty}$ class relative to $\nu$. Similarly, if $\mu,\nu$ are Radon measures on $\R$, both globally doubling, and the Carleson condition \eqref{carlesonADT} holds for the second powers $\alpha^{2}_{\mu,\nu}(B(y,t))$, then $\mu \in A_{\infty}(\nu)$.
\end{thm}

The \emph{a priori} doubling assumptions cannot be omitted (that is, they are not implied by the Carleson condition): just consider $\mu = 2\chi_{[0,1/2)} \, d\calL$. It is clear that the Carleson condition \eqref{carleson} holds for the numbers $\alpha^{2}_{\mu,\calL}(I)$, but nevertheless $\mu \notin A_{\infty}^{\calD}(\calL|_{[0,1]})$.

\subsection{Outline of the paper, and the main steps of the proofs} The main substance of the paper is proving the dyadic result, Theorem \ref{main}, and in particular part (b). This work takes up Sections \ref{deltaAlphaComparison}-\ref{mainProof}. The proof of part (a) is simpler, and closely follows a previous argument of Tolsa -- namely the one used to prove \eqref{finiteSF}. The details (both in the dyadic and continuous settings) are given in Section \ref{muASection}. Modifications required to prove part (b) of the "continuous" Theorem \ref{mainCont} are outlined in Section \ref{continuousVariants}.

 The proof of Theorem \ref{main}(b) has three main steps. First, the numbers $\alpha_{\mu,\nu}(I)$ are used to control something analyst-friendlier, namely the following dyadic variants:
\begin{equation}\label{deltaNumbers} \Delta_{\mu,\nu}(I) = \left|\frac{\mu(I_{-})}{\mu(I)} - \frac{\nu(I_{-})}{\nu(I)} \right|. \end{equation}
Here $I_{-}$ stands for the left half of $I$. This would be simple, if $\chi_{[0,1/2)}$ happened to be one of the admissible test functions $\psi$ in the definition of $\W_{1}$. It is not, however, and in fact there seems to be no direct (and sufficiently efficient) way to control $\Delta_{\mu,\nu}(I)$ by $\alpha_{\mu,\nu}(I)$, or even $\alpha_{\mu,\nu}(3I)$. However, it turns out that the numbers are equivalent at the level of certain Carleson sums over trees; proving this statement is the main content of Section \ref{deltaAlphaComparison}.

The numbers $\Delta_{\mu,\nu}(I)$ are well-known quantities: they are the (absolute values of the) coefficients in an orthogonal representation of $\mu$ in terms of $\nu$-adapted Haar functions, and it is known that they can be used to characterise $A_{\infty}$. The following theorem is due to S. Buckley \cite{Bu} from 1993:
\begin{thm}[Theorem 2.2(iii) in \cite{Bu}]\label{Bu} Let $\mu,\nu$ be a dyadically doubling Borel probability measures on $[0,1]$. Then $\mu \in A_{\infty}^{\calD}(\nu)$, if and only if
\begin{equation}\label{carlesonBu} \sum_{I \subset J} \Delta^{2}_{\mu,\nu}(I)\mu(I) \leq C\mu(J), \qquad J \in \calD. \end{equation}
\end{thm}
The result in \cite{Bu} is only stated for $\nu = \calL|_{[0,1]}$, but the proof works in the greater generality. Note the similarity between the Carleson conditions \eqref{carlesonBu} and \eqref{carleson}: The dyadic part of Theorem \ref{mainCarleson} is, in fact, nothing but a corollary of Buckley's result, assuming that one knows how to control the numbers $\Delta_{\mu,\nu}(I)$ by the numbers $\alpha_{\mu,\nu}(I)$ at the level of Carleson sums; consequently, the short proof of this half of Theorem \ref{mainCarleson} can be found in Section \ref{deltaAlphaComparison}. The continuous version is discussed briefly in Remark \ref{AinftyRemark}.

Buckley's result is not applicable for Theorem \ref{main}: the measure $\mu$ is not dyadically doubling, and the information available is much weaker than the Carleson condition \eqref{carleson}. Handling these issues constitutes the remaining two steps in the proof: all dyadic intervals are split into trees, where $\mu$ is "tree-doubling" (Section \ref{mainProof}), and the absolute continuity of $\mu$ with respect to $\nu$ is studied in each tree separately (Section \ref{treeAbsCont}). 

\subsection{Acknowledgements} I am grateful to Jonas Azzam, David Bate, and Antti K\"aenmaki and for useful discussions during the preparation of the manuscript. I would also like to thank the referees for good comments, and for asking me to prove parts (a) of Theorems \ref{main} and \ref{mainCont}.

\section{Comparison of $\alpha$-numbers and $\Delta$-numbers}\label{deltaAlphaComparison}

In this section, $\mu$ and $\nu$ are Borel probability measures on $[0,1)$, $\mu$ does not charge the boundaries of dyadic intervals, and $\nu$ is dyadically doubling inside $[0,1)$:
\begin{displaymath} \nu(\hat{I}) \leq D_{\nu}\nu(I), \qquad I \in \calD \setminus \{[0,1)\}. \end{displaymath}
This implies, in particular, that $\nu(I) > 0$ for all $I \in \calD$ with $I \subset [0,1)$. The main task of the section is to bound the numbers $\Delta_{\mu,\nu}(I)$ by the numbers $\alpha_{\mu,\nu}(I)$, where $\Delta_{\mu,\nu}(I)$ was the quantity
\begin{displaymath} \Delta_{\mu,\nu}(I) = \left|\frac{\mu(I_{-})}{\mu(I)} - \frac{\nu(I_{-})}{\nu(I)} \right| = \left| \int \chi_{(0,1/2)} \, d\mu_{I} - \int \chi_{(0,1/2)} \, d\nu_{I} \right|. \end{displaymath}
The task would be trivial, if $\chi_{(0,1/2)}$ were a $1$-Lipschitz function vanishing at the boundary of $[0,1]$. It is not: in fact, the difference between $\Delta_{\nu_{1},\nu_{2}}(I)$ and $\alpha_{\nu_{1},\nu_{2}}(I)$ can be rather large for a given interval $I$.

\begin{ex}\label{exampleOne} If $\nu_{1} = \delta_{1/2 - 1/n}$ and $\nu_{2} = \delta_{1/2 + 1/n}$, then $\Delta_{\nu_{1},\nu_{2}}([0,1)) = 1$, but $\alpha_{\nu_{1},\nu_{2}}([0,1)) \lesssim 1/n$. These measures do not satisfy the assumptions of the section, so consider also the following example. Let $\mu = f \, d\calL$, where $f$ takes the value $1$ everywhere, except in the $2^{-n}$-neighbourhood of $1/2$. Let $f \equiv 1/2$ on the interval $[1/2 - 2^{-n},1/2]$, and $f \equiv 3/2$ on the interval $(1/2,1/2 + 2^{-n}]$. Then $\mu$ is dyadically $4$-doubling probability measure on $[0,1]$, $\Delta_{\mu,\calL}([0,1)) \sim 2^{-n}$, and $\alpha_{\mu,\calL}([0,1)) \sim 2^{-2n}$. \end{ex}

Fortunately, "pointwise" estimates between $\Delta_{\mu,\nu}(I)$ and $\alpha_{\mu,\nu}(I)$ are not really needed in this paper, and it turns out that certain sums of these numbers are comparable, up to a manageable error. To state such results, I need to introduce some terminology. A family $\mathcal{C} \subset \calD$ of dyadic intervals is called \emph{coherent}, if the implication
\begin{displaymath} Q,R \in \mathcal{C} \text{ and } Q \subset P \subset R \quad \Longrightarrow \quad P \in \mathcal{C} \end{displaymath}
holds for all $Q,P,R \in \calD$. 
\begin{definition}[Trees, leaves, boundary]\label{treeDef} A \emph{tree} $\calT \subset \calD$ is any coherent family of dyadic intervals with a unique largest interval, $\Top(\calT) \in \calT$, and with the property that 
\begin{displaymath} \card (\ch(I) \cap \calT) \in \{0,2\}, \qquad I \in \calT. \end{displaymath}
For the tree $\calT$, define the set family $\Leaves(\calT)$ to consist of the minimal intervals in $\calT$, in other words those $I \in \calT$ with $\card(\ch(I) \cap \calT) = 0$. Abusing notation, I often write $\Leaves(\calT)$ also for the set $\cup\{I : I \in \Leaves(\calT)\}$. Finally, define the \emph{boundary} of the tree $\partial \calT$ by
\begin{displaymath} \partial T := \Top(\calT) \setminus \Leaves(\calT). \end{displaymath}
Then $x \in \partial \calT$, if and only if $x \in \Top(\calT)$, and all intervals $I \in \calD$ with $x \in I \subset \Top(\calT)$ are contained $\calT$. \end{definition}

\begin{definition}[$(\calT,D)$-doubling measures]\label{treeDoubling} A Borel probability measure $\mu$ on $[0,1]$ is called $(\calT,D)$-doubling, if 
\begin{displaymath} \mu(\hat{I}) \leq D\mu(I), \qquad I \in \calT \setminus \Top(\calT). \end{displaymath}
\end{definition}

Here is the main result of this section:
\begin{proposition}\label{deltaVsAlpha} Let $\mu,\nu$ be measures satisfying the assumptions of the section, and let $\calT \subset \calD$ be a tree. Moreover, assume that $\mu$ is $(\calT,D)$-doubling for some constant $D \geq 1$. Then
\begin{displaymath} \sum_{I \in \calT} \Delta_{\mu,\nu}^{2}(I)\mu(I) \lesssim_{D_{\nu},D} \sum_{I \in \calT \setminus \Leaves(\calT)} \alpha_{\mu,\nu}^{2}(I)\mu(I) + \mu(\Top(\calT)). \end{displaymath}
\end{proposition}

The "dyadic part" of Theorem \ref{mainCarleson} is an immediate corollary:
\begin{proof}[Proof of Theorem \ref{mainCarleson}, dyadic part] By hypothesis, both measures $\mu$ and $\nu$ are $(\calD,C)$-doubling. Hence, by the Carleson condition \eqref{carleson}, and Proposition \ref{deltaVsAlpha} applied to the trees $\calT_{J} := \{I \in \calD : I \subset J\}$, one has
\begin{displaymath} \sum_{I \subset J} \Delta_{\mu,\nu}^{2}(I)\mu(I) \lesssim_{C} \sum_{I \subset J} \alpha_{\mu,\nu}^{2}(I)\mu(I) + \mu(J) \lesssim \mu(J). \end{displaymath}
This is precisely the condition in Buckley's result, Theorem \ref{Bu}, so $\mu \in A^{\calD}_{\infty}(\nu)$.  \end{proof}

I then begin the proof of Proposition \ref{deltaVsAlpha}. It would, in fact, suffice to assume that $\nu$ is also just $(\calT,D_{\nu})$-doubling, but checking this would result in some unnecessary book-keeping below. The proof is based on the observation that $\chi_{(0,1/2)}$ can be written as a series of Lipschitz functions, each supported on sub-intervals of $[0,1]$. This motivates the following considerations.

Assume that 
\begin{displaymath} \Psi := \Psi_{0} := \sum_{j \geq 0} \psi_{j} \end{displaymath}
is a bounded function such that each $\psi_{j} \colon \R \to [0,\infty)$ is an $L_{j}$-Lipschitz function supported on some interval $I_{j} \in \calD_{j}$. Assume moreover that the intervals $I_{j}$ are nested: $[0,1) \supset I_{1} \supset I_{2} \ldots$. Then, as a first step in proving Proposition \ref{deltaVsAlpha}, I claim that
\begin{align}\label{representation+} \left| \int \Psi \, d\mu - \int \Psi \, d\nu \right| & \leq \sum_{k = 0}^{N} \frac{L_{k}}{2^{k}}\alpha_{\mu,\nu}(I_{k})\mu(I_{k})\\
& + \sum_{k = 0}^{N} \left(\frac{1}{\nu(I_{k + 1})} \int \Psi_{k + 1} \, d\nu \right) \Delta_{\mu,\nu}(I_{k})\mu(I_{k}) + 2\|\Psi\|_{\infty} \mu(I_{N + 1})\notag \end{align}
for any $N \in \{0,1,\ldots,\infty\}$, where
\begin{displaymath} \Psi_{k} := \sum_{j \geq k} \psi_{j}, \qquad m \geq 0. \end{displaymath}
For $N = \infty$, the symbol "$I_{N + 1}$" should be interpreted as the intersection of all the intervals $I_{j}$. I will first verify that, for any $m \geq 0$,
\begin{align} & \left| \frac{1}{\mu(I_{m})} \int \Psi_{m} \, d\mu - \frac{1}{\nu(I_{m})} \int \Psi_{m} \, d\nu \right| \notag\\
&\label{repInd} \leq \frac{L_{m}}{2^{m}}\alpha_{\mu,\nu}(I_{m}) + \left(\frac{1}{\nu(I_{m + 1})} \int \Psi_{m + 1} \, d\nu \right) \Delta_{\mu,\nu}(I_{m})\\
&\quad + \frac{\mu(I_{m + 1})}{\mu(I_{m})}\left| \frac{1}{\mu(I_{m + 1})} \int \Psi_{m + 1} \, d\mu - \frac{1}{\nu(I_{m + 1})} \int \Psi_{m + 1} \, d\nu \right| \notag \end{align}
from which it will be easy to derive \eqref{representation+}. If $\mu(I_{m}) = 0$, the corresponding term should be interpreted as "$0$" (recall that $\nu(I_{m})$ is never zero by the doubling assumption). The proof of \eqref{repInd} is straightforward. First, note that since $\psi_{m} \colon \R \to \R$ is an $L_{m}$-Lipschitz function supported on $I_{m}$, and $|I_{m}| = 2^{-m}$, one has
\begin{displaymath} \left| \frac{1}{\mu(I_{m})} \int \psi_{m} \, d\mu - \frac{1}{\nu(I_{m})} \int \psi_{m} \, d\nu \right| = \left| \int \psi_{m} \circ T_{I_{m}}^{-1} \, d\mu_{I_{m}} - \int \psi_{m} \circ T_{I_{m}}^{-1} \, d\nu_{I_{m}} \right| \leq \frac{L_{m}}{2^{m}}\alpha_{\mu,\nu}(I_{m}). \end{displaymath}
(The mappings $T_{I}$ are familiar from Definition \ref{alphas}). This gives rise to the first term in \eqref{repInd}. What remains is bounded by
\begin{align*} & \left| \frac{1}{\mu(I_{m})} \int \Psi_{m + 1} \, d\mu - \frac{1}{\nu(I_{m})} \int \Psi_{m + 1} \, d\nu \right|\\
& \leq \frac{\mu(I_{m + 1})}{\mu(I_{m})} \left| \frac{1}{\mu(I_{m + 1})} \int \Psi_{m + 1} \, d\mu - \frac{1}{\nu(I_{m + 1})} \int \Psi_{m + 1} \, d\nu \right|\\
&\quad + \left(\frac{1}{\nu(I_{m + 1})} \int \Psi_{m + 1} \, d\nu \right) \left| \frac{\mu(I_{m + 1})}{\mu(I_{m})} - \frac{\nu(I_{m + 1})}{\nu(I_{m})} \right|.   \end{align*} 
This is \eqref{repInd}, observing that
\begin{displaymath} \Delta_{\mu,\nu}(I_{m}) = \left| \frac{\mu(I_{m + 1})}{\mu(I_{m})} - \frac{\nu(I_{m + 1})}{\nu(I_{m})} \right|, \end{displaymath}
since either $I_{m + 1} = (I_{m})_{+}$ or $I_{m + 1} = (I_{m})_{-}$, and both possibilities give the same number $\Delta_{\mu,\nu}(I_{m})$. Finally, \eqref{representation+} is obtained by repeated application of \eqref{repInd}. By induction, one can check that $N$ iterations of \eqref{repInd} (starting from $m = 0$, and recalling that $\mu,\nu$ are probability measures on $[0,1)$) leads to
\begin{align} & \left| \int \Psi \, d\mu - \int \Psi \, d\nu \right| \leq \sum_{k = 0}^{N} \frac{L_{k}}{2^{k}}\alpha_{\mu,\nu}(I_{k})\mu(I_{k}) + \sum_{k = 0}^{N} \left(\frac{1}{\nu(I_{k + 1})} \int \Psi_{k + 1} \, d\nu \right) \Delta_{\mu,\nu}(I_{k})\mu(I_{k})\notag \\
&\label{form17} \qquad + \mu(I_{N + 1}) \left|\frac{1}{\mu(I_{N + 1})} \int \Psi_{N + 1} \, d\mu - \frac{1}{\nu(I_{N + 1})} \int \Psi_{N + 1} \, d\nu \right|.  \end{align}
This gives \eqref{representation+} immediately, observing that $\|\Psi_{N + 1}\|_{\infty} \leq \|\Psi\|_{\infty}$. 

Now, it is time to specify the functions $\psi_{j}$. I first define a hands-on Whitney decomposition for $(0,1/2)$. Pick a small parameter $\tau > 0$, to be specified later, and let $U_{0} := [\tau,1/2 - \tau)$. Then, set $U_{-k} := [\tau 2^{-k},\tau 2^{-k + 1})$ and $U_{k} := 1/2 - U_{-k}$ for $k \geq 1$. Let $\{\psi_{k}\}_{k \in \Z}$ be a partition of unity subordinate to slightly enlarged versions of the sets $U_{k}$, $k \in \Z$. By this, I first mean that each $\psi_{k}$ is non-negative and $L_{k}$-Lipschitz with 
\begin{equation}\label{LipschitzBound} L_{k} \leq \frac{C2^{|k|}}{\tau}. \end{equation}
Second, the supports of the functions $\psi_{k}$ should satisfy $\psi_{0} \subset [\tau/2,1/2 - \tau/2)$, 
\begin{displaymath} \spt \psi_{-k} \subset [(\tau/2) 2^{-k},2\tau 2^{-k + 1}) \subset (0,2\tau 2^{-k + 1}) \quad \text{and} \quad \psi_{k} \subset 1/2 - (0,2\tau 2^{-k + 1}) \end{displaymath}
for $k \geq 1$. Third,
\begin{displaymath} \sum_{k \in \Z} \psi_{k} = \chi_{(0,1/2)}. \end{displaymath}
Let $\Psi^{-} := \sum_{k > 0} \psi_{-k} + \psi_{0}/2$ and $\Psi^{+} := \sum_{k > 0} \psi_{k} + \psi_{0}/2$. Then
\begin{equation}\label{piece2} \Delta_{\mu,\nu}([0,1)) \leq \left| \int \Psi^{-} \, d\mu - \int \Psi^{-} \, d\nu \right| + \left|\int \Psi^{+} \, d\mu - \int \Psi^{+} \, d\nu \right|. \end{equation}
This is the only place in the paper, where the assumption of $\mu$ not charging the boundaries of dyadic intervals is used (however, the estimate \eqref{piece2} will eventually be applied to all the measures $\mu_{I}$, $I \in \calD$, so the full strength of the hypothesis is needed). The function $\Psi^{-}$ is precisely of the form treated above with $I_{j} := [0,2^{-j})$, since clearly $\spt \psi_{-k} \subset I_{k}$. Applying the inequality \eqref{representation+} with any $N_{1} \in \{0,1,\ldots,\infty\}$ yields
\begin{align}\label{form18} \left| \int \Psi^{-} \, d\mu - \int \Psi^{-} \, d\nu \right| & \leq \sum_{k = 0}^{N_{1}} \frac{L_{-k}}{2^{k}}\alpha_{\mu,\nu}(I_{k})\mu(I_{k})\\
& + \sum_{k = 0}^{N_{1}} \left(\frac{1}{\nu(I_{k + 1})} \int \Psi^{-}_{k + 1} \, d\nu \right) \Delta_{\mu,\nu}(I_{k})\mu(I_{k}) + 2\mu(I_{N_{1} + 1}).\notag \end{align}
Next, observe that each function $\Psi^{-}_{k + 1}$, $k \geq 0$, is bounded by $1$ and vanishes outside
\begin{displaymath} \bigcup_{j = k + 1}^{\infty} \spt \psi_{-k} \subset (0,2\tau 2^{-k}). \end{displaymath}
It follows that
\begin{displaymath} \frac{1}{\nu(I_{k + 1})} \int \Psi^{-}_{k + 1} \, d\nu \leq \frac{\nu((0,2\tau 2^{-k}))}{\nu(I_{k + 1})} = o_{D_{\nu}}(\tau), \end{displaymath}
where the implicit constants only depend on the dyadic doubling constant $D_{\nu}$ of $\nu$. In the sequel, I assume that $\tau$ is so small that $o_{D_{\nu}}(\tau) \leq \kappa$, where $\kappa > 0$ is another small constant, which will eventually depend on the $(\calT,D)$-doubling constant $D$ for $\mu$. Recalling also \eqref{LipschitzBound}, the estimate \eqref{form18} then becomes
\begin{equation}\label{piece1} \left| \int \Psi^{-} \, d\mu - \int \Psi^{-} \, d\nu \right| \leq \frac{C}{\tau} \sum_{k = 0}^{N_{1}} \alpha_{\mu,\nu}(I_{k})\mu(I_{k}) + \kappa \sum_{k = 0}^{N_{1}} \Delta_{\mu,\nu}(I_{k})\mu(I_{k}) + 2\mu(I_{N_{1} + 1}). \end{equation}
The last term simply vanishes, if $N_{1} = \infty$, because $\mu(\{0\}) = 0$. A heuristic point to observe is that the left hand side is roughly $\Delta_{\mu,\nu}([0,1])$; the right hand side also contains the same term, but multiplied by a small constant $\kappa > 0$. This gain is "paid for" by the large constant $C/\tau$. 

Next, the estimate is replicated for $\Psi^{+}$. This time, the inequality \eqref{representation+} is applied to the sequence $\tilde{I}_{0} = [0,1)$, $\tilde{I}_{1} = [0,1/2)$, $\tilde{I}_{2} = (\tilde{I}_{1})_{+}$, and in general $\tilde{I}_{k + 1} = (\tilde{I}_{k})_{+}$ for $k \geq 1$ (here $J_{+}$ is the right half of $J$). Then, if $\tau$ is small enough, it is again clear that $\spt \psi_{k} \subset \tilde{I}_{k}$. Thus, by inequality \eqref{representation+},
\begin{align}\label{form26} \left| \int \Psi^{+} \, d\mu - \int \Psi^{+} \, d\nu \right| & \leq \sum_{k = 0}^{N_{2}} \frac{L_{k}}{2^{k}}\alpha_{\mu,\nu}(\tilde{I}_{k})\mu(\tilde{I}_{k})\\
& + \sum_{k = 0}^{N_{2}} \left(\frac{1}{\nu(\tilde{I}_{k + 1})} \int \Psi^{+}_{k + 1} \, d\nu \right) \Delta_{\mu,\nu}(\tilde{I}_{k})\mu(\tilde{I}_{k}) + 2\mu(\tilde{I}_{N_{2} + 1})\notag \end{align}
for any $N_{2} \geq 0$. As before, the term $\mu(\tilde{I}_{N_{2}})$ vanishes for $N_{2} = \infty$ (because $\mu(\{\tfrac{1}{2}\}) = 0$), and one can ensure
\begin{displaymath} \frac{1}{\nu(\tilde{I}_{k + 1})} \int \Psi^{+}_{k + 1} \, d\nu \leq \kappa \end{displaymath}
by choosing $\tau = \tau(D_{\nu}) > 0$ small enough. Consquently (recalling \eqref{piece2}), \eqref{piece1} and \eqref{form26} together imply 
\begin{equation}\label{CalF} \Delta_{\mu,\nu}([0,1)) \leq \frac{C}{\tau} \sum_{I \in \Tail} \alpha_{\mu,\nu}(I)\mu(I) + \kappa \sum_{I \in \Tail} \Delta_{\mu,\nu}(I)\mu(I) + 2\mu(I_{N_{1} + 1}) + 2\mu(\tilde{I}_{N_{2} + 1}). \end{equation}
Here $\Tail$ is the collection of all the intervals $I_{0},\ldots,I_{N_{1}}$ and $\tilde{I}_{0},\ldots,\tilde{I}_{N_{2}}$. The intervals $[0,1)$ and $[0,1/2)$ arise a total of two times from \eqref{piece1} and \eqref{form26}, but this has no visible impact on the end result, \eqref{CalF}. 
\begin{figure}[h!]
\begin{center}
\includegraphics[scale = 0.6]{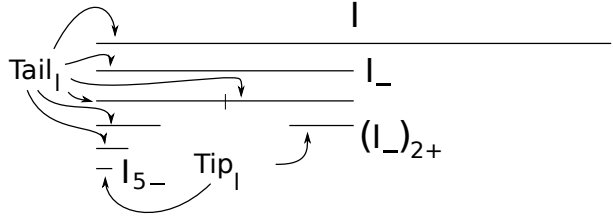}
\caption{An example of $\Tail_{I}(4,1)$ and $\Tip_{I}$.}\label{fig1}
\end{center}
\end{figure}
The estimate \eqref{CalF} generalises in a simple way to other intervals $I \in \calD$, besides $I = [0,1)$, but requires an additional piece of notation. Let $I \in \calD$, and write $I_{0-} := I =: I_{0+}$. For $k \geq 1$, define $I_{k-} := (I_{(k - 1)-})_{-}$ and $I_{k+} := (I_{(k - 1)+})_{+}$. Now, for a fixed dyadic interval $I \subset [0,1)$, and $N_{1},N_{2} \geq 0$, let $\Tail_{I} = \Tail_{I}(N_{1},N_{2})$ be the collection of subintervals of $I$, which includes $I_{k-}$ for all $0 \leq k \leq N_{1}$ and $(I_{-})_{k+}$ for all $0 \leq k \leq N_{2}$, see Figure \ref{fig1}. Then, the generalisation of \eqref{CalF} reads
\begin{equation}\label{CalF+} \Delta_{\mu,\nu}(I)\mu(I) \leq \frac{C}{\tau} \sum_{J \in \Tail_{I}} \alpha_{\mu,\nu}(J)\mu(J) + \kappa \sum_{J \in \Tail_{I}} \Delta_{\mu,\nu}(J)\mu(J) + 2\mu(\Tip_{I}), \end{equation}
where $\Tip_{I} = I_{(N_{1} + 1)-} \cup (I_{-})_{(N_{2} + 1)+}$. If $N_{1} < \infty$ and $N_{2} = \infty$, for instance, then $\Tip_{I} = I_{(N_{1} + 1)-}$. The proof is nothing but an application of \eqref{CalF} to the measures $\mu_{I}$ and $\nu_{I}$. For minor technical reasons, I also wish to allow the choice $N_{1} = 0$ and $N_{2} = -1$: by definition, this choice means that $\Tail_{I} = \{I\}$ and $\Tip_{I} := I_{-}$. It is easy to see that \eqref{CalF+} remains valid in this case, with "$2$" replaced by "$4$" (for $I = [0,1)$, this follows by applying \eqref{piece1} and \eqref{form26} with the choices $N_{1} = 0 = N_{2}$).

Now, the table is set to prove Proposition \ref{deltaVsAlpha}, which I recall here:
\begin{proposition} Let $\mu,\nu$ be measures satisfying the assumptions of the section, and let $\calT \subset \calD$ be a tree. Moreover, assume that $\mu$ is $(\calT,D)$-doubling for some constant $D \geq 1$. Then
\begin{displaymath} \sum_{I \in \calT} \Delta_{\mu,\nu}^{2}(I)\mu(I) \lesssim_{D_{\nu},D} \sum_{I \in \calT \setminus \Leaves(\calT)} \alpha_{\mu,\nu}^{2}(I)\mu(I) + \mu(\Top(\calT)). \end{displaymath}
\end{proposition}

\begin{proof} The sum over $I \in \Leaves(\calT)$ is evidently bounded by $4\mu(\Top(\calT))$, so it suffices to consider 
\begin{displaymath} I \in \calT \setminus \Leaves(\calT) =: \calT^{-}. \end{displaymath}
Let $I \in \calT$, and define the number $N_{1} = N_{1}(I) \geq 0$ as the smallest index so that $I_{(N_{1} + 1)-} \in \Leaves(\calT)$. If no such index exists, set $N_{1} = \infty$. If $I_{-} \in \Leaves(\calT)$, then $N_{1} = 0$, and I define $N_{2} = -1$: then $\Tail_{I} := \{I\}$, and $\Tip_{I} := I_{-}$. Otherwise, if $I_{-} \in \calT^{-}$, let $N_{2} \geq 0$ be the smallest index such that $(I_{-})_{(N_{2} + 1)+} \in \Leaves(\calT)$. If no such index exists, let $N_{2} = \infty$. Now $\Tail_{I} \subset \calT^{-}$ and $\Tip_{I} \subset \Leaves(\calT)$ are defined as after \eqref{CalF+}. Start by the following combination of \eqref{CalF+} and Cauchy-Schwarz:
\begin{align} \Delta_{\mu,\nu}^{2}(I)\mu(I)^{2} & \lesssim \frac{1}{\tau^{2}} \left( \sum_{J \in \Tail_{I}} \alpha_{\mu,\nu}^{2}(J)\mu(J)^{3/2} \right)\left(\sum_{J \in \Tail_{I}} \mu(J)^{1/2} \right)\notag\\
&\label{form19} + \kappa^{2} \left(\sum_{J \in \Tail_{I}} \Delta_{\mu,\nu}^{2}(J)\mu(J)^{3/2} \right) \left(\sum_{J \in \Tail_{I}} \mu(J)^{1/2} \right) + \mu(\Tip_{I})^{2}. \end{align} 
The factors $\sum_{J \in \Tail_{I}} \mu(J)^{1/2}$ are under control, thanks to the $(\calT,D)$-doubling hypothesis on $\mu$, and the fact that $\Tail_{I} \subset \calT$. Since $\Tail_{I}$ consists of two "branches" of nested intervals inside $I$, and the $(\calT,D)$-doubling hypothesis implies that the $\mu$-measures of intervals decay geometrically along these branches, one arrives at
\begin{displaymath} \sum_{J \in \Tail_{I}} \mu(J)^{1/2} \lesssim_{D} \mu(I)^{1/2}. \end{displaymath}
Thus, by \eqref{form19},
\begin{equation}\label{form20} \Delta_{\mu,\nu}^{2}(I)\mu(I) \lesssim_{D}\frac{1}{\tau^{2}} \sum_{J \in \Tail_{I}} \alpha_{\mu,\nu}^{2}(J)\frac{\mu(J)^{3/2}}{\mu(I)^{1/2}} + \kappa^{2} \sum_{J \in \Tail_{I}} \Delta_{\mu,\nu}^{2}(J)\frac{\mu(J)^{3/2}}{\mu(I)^{1/2}} + \frac{\mu(\Tip_{I})^{2}}{\mu(I)}. \end{equation}
The constant $\kappa > 0$ will have to be chosen so small, eventually, that its product with the implicit constants above is notably less than one. From now on, the precise restriction $J \in \Tail_{I}$ can be replaced by the conditions $J \in \calT^{-}$ and $J \subset I$. With this in mind, observe first that
\begin{align*} \sum_{I \in \calT^{-}} \mathop{\sum_{J \in \calT^{-}}}_{J \subset I} \alpha_{\mu,\nu}^{2}(J) \frac{\mu(J)^{3/2}}{\mu(I)^{1/2}} & = \sum_{J \in \calT^{-}} \alpha_{\mu,\nu}^{2}(J)\mu(J)^{3/2} \mathop{\sum_{I \in \calT^{-}}}_{I \supset J} \frac{1}{\mu(I)^{1/2}}\\
& \lesssim_{D} \sum_{J \in \calT^{-}} \alpha_{\mu,\nu}^{2}(J)\mu(J).  \end{align*} 
The final inequality uses, again, the geometric decay of $\mu$-measures of intervals in $\calT$. A similar estimate can be performed for the second term in \eqref{form20}. As for the third term,
\begin{align*} \sum_{I \in \calT^{-}} \frac{\mu(\Tip_{I})^{2}}{\mu(I)} & \lesssim \sum_{I \in \calT^{-}} \frac{\mu(I_{(N_{1} + 1)-})^{2} + \mu((I_{-})_{(N_{2} + 1)+})^{2}}{\mu(I)}\\
& \lesssim \sum_{J \in \Leaves(\calT)} \mu(J)^{2} \mathop{\sum_{I \in \calT^{-}}}_{I \supset J} \frac{1}{\mu(I)} \lesssim_{D} \mu(\Leaves(\calT)), \end{align*}
relying once more on the geometric decay of $\mu$ in $\calT$. 
Combining all the estimates gives
\begin{equation}\label{form21} \sum_{I \in \calT^{-}} \Delta_{\mu,\nu}^{2}(I)\mu(I) \lesssim_{D} \frac{1}{\tau^{2}} \sum_{I \in \calT^{-}} \alpha_{\mu,\nu}^{2}(I)\mu(I) + \kappa^{2} \sum_{I \in \calT^{-}} \Delta_{\mu,\nu}^{2}(I)\mu(I) + \mu(\Leaves(\calT)). \end{equation}
If the left hand side is \emph{a priori} finite, the proof of Proposition \ref{deltaVsAlpha} is now completed by choosing $\kappa$ small enough, depending on $D$. If not, consider any finite sub-tree $\calT_{j} \subset \calT$ with $\Top(\calT_{j}) = \Top(\calT)$. Then, the proof above gives \eqref{form21} with $\calT_{j}$ in place of $\calT$. Hence
\begin{displaymath} \sum_{I \in \calT_{j}^{-}} \Delta_{\mu,\nu}^{2}(I)\mu(I) \lesssim_{D} \sum_{I \in \calT_{j}^{-}} \alpha_{\mu,\nu}^{2}(I)\mu(I) + \mu(\Top(\calT)), \end{displaymath}
where the constants do not depend on the choice of $\calT_{j}$. Now the proposition follows by letting $\calT_{j} \nearrow \calT$. 
\end{proof}

\section{Absolute continuity of tree-adapted measures}\label{treeAbsCont} Recall the concepts of tree, leaves and boundaries from Definition \ref{treeDef}, and the notion of $(\calT,D)$-doubling measures from Definition \ref{treeDoubling}. In the present section, I assume that $\calT \subset \calD$ is a tree, and $\mu,\nu$ are two finite Borel measures, which satisfy the following two assumptions:
\begin{itemize}
\item[(A)] $\min\{\mu(\Top(\calT)),\nu(\Top(\calT))\} > 0$, and
\item[(B)] $\mu,\nu$ are $(\calT,D)$-doubling for some constant $D \geq 1$.
\end{itemize}
In particular, the assumptions imply that
\begin{displaymath}\mu(I) > 0 \quad \text{and} \quad \nu(I) > 0, \qquad I \in \calT. \end{displaymath}
For reasons to become apparent soon, I define the \emph{$(\calT,\mu)$-adaptation of $\nu$},
\begin{displaymath} \nu_{\calT} := \nu|_{\partial \calT} + \sum_{I \in \Leaves(\calT)} \tfrac{\nu}{\mu}(I) \cdot \mu|_{I}, \end{displaymath}
where $\tfrac{\nu}{\mu}(I) := \nu(I)/\mu(I)$. Note that
\begin{equation}\label{form8+} \nu_{\calT}(I) = \nu(I), \qquad I \in \calT, \end{equation}
because $\partial \calT$ is disjoint from the leaves, which are also pairwise disjoint. In particular, $\nu_{\calT}(\Top(\calT)) = \nu(\Top(\calT))$. The main result of the section is the following:
\begin{proposition}\label{mainProp} Assume \emph{(A)} and \emph{(B)}, and that
\begin{displaymath} \sum_{I \in \calT \setminus \Leaves(\calT)} \Delta_{\mu,\nu}^{2}(I)\mu(I) < \infty. \end{displaymath}
Then $\mu|_{\Top(\calT)} \ll \nu_{\calT}$. In particular $\mu|_{\partial \calT} \ll \nu$.
\end{proposition} 

\begin{remark} By the definition of $\nu_{\calT}$, it is obvious that $\mu|_{\Leaves(\calT)} \ll \nu_{\calT}$. So, the main point of Proposition \ref{mainProp} is to show that $\mu|_{\partial \calT} \ll (\nu_{\calT})|_{\partial \calT} = \nu|_{\partial \calT}$. \end{remark}

Since $\mu(\Top(\calT)) > 0$ and $\nu(\Top(\calT)) > 0$, one may assume without loss of generality that
\begin{displaymath} \mu(\Top(\calT)) = 1 = \nu(\Top(\calT)). \end{displaymath}
The proof of Proposition \ref{mainProp} is based on a "product representation" for $\nu_{\calT}$, relative to $\mu$, in the spirit of \cite[Theorem 3.22]{FKP} of Fefferman, Kenig and Pipher. Recall that every interval $I \in \calD$ has exactly two children: $I_{-}$ and $I_{+}$. Define the $\mu$-adapted Haar functions
\begin{displaymath} h_{I}^{\mu} := c_{I}^{+}\chi_{I_{+}} - c_{I}^{-}\chi_{I_{-}}, \qquad I \in \calT \setminus \Leaves(\calT), \end{displaymath}
where
\begin{displaymath} c_{I}^{+} := \frac{\mu(I)}{\mu(I_{+})} \quad \text{and} \quad c_{I}^{-} := \frac{\mu(I)}{\mu(I_{-})}. \end{displaymath}
This ensures that $\int h_{I}^{\mu} \, d\mu = 0$ for $I \in \calT \setminus \Leaves(\calT)$. Note that $\mu(I_{+}),\mu(I_{-}) > 0$, because $I_{+},I_{-} \in \calT$. Now, the plan is to define coefficients $a_{J} \in \R$, for $J \in \calT \setminus \Leaves(\calT)$, so that the following requirement is met:
\begin{equation}\label{form1} \mathop{\prod_{J \supsetneq I}}_{J \in \calT} (1 + a_{J}h_{J}^{\mu})(x) = \tfrac{\nu}{\mu}(I), \qquad x \in I \in \calT. \end{equation}
The left hand side of \eqref{form1} is certainly constant on $I$, so the equation has some hope; if $I = \Top(\calT)$, then the product is empty, and the right hand side of \eqref{form1} equals $1$ by the assumption $\mu(\Top(\calT)) = \nu(\Top(\calT)) = 1$. Now, assume that \eqref{form1} holds for some interval $I \in \calT \setminus \Leaves(\calT)$. Then $I_{-},I_{+} \in \calT$, so if \eqref{form1} is supposed to hold for $I_{-}$, one has
\begin{equation}\label{form2} \tfrac{\nu}{\mu}(I_{-}) = \mathop{\prod_{J \supsetneq I_{-}}}_{J \in \calT} (1 + a_{J}h_{J}^{\mu}) = (1 - c_{I}^{-}a_{I})\mathop{\prod_{J \supsetneq I}}_{J \in \calT} (1 + a_{J}h_{J}^{\mu}) = (1 - c_{I}^{-}a_{I})\tfrac{\nu}{\mu}(I), \end{equation}
and similarly
\begin{equation}\label{form3} \tfrac{\nu}{\mu}(I_{+}) = (1 + c_{I}^{+}a_{I})\tfrac{\nu}{\mu}(I). \end{equation}
From \eqref{form2} one solves
\begin{equation}\label{form4} a_{I} = \frac{\tfrac{\nu}{\mu}(I) - \tfrac{\nu}{\mu}(I_{-})}{\tfrac{\nu}{\mu}(I)c_{I}^{-}} = \frac{\mu(I_{-})}{\mu(I)} - \frac{\nu(I_{-})}{\nu(I)}, \end{equation}
and \eqref{form3} gives
\begin{equation}\label{form5} a_{I} = \frac{\tfrac{\nu}{\mu}(I_{+}) - \tfrac{\nu}{\mu}(I)}{\tfrac{\nu}{\mu}(I)c_{I}^{+}} = \frac{\nu(I_{+})}{\nu(I)} - \frac{\mu(I_{+})}{\mu(I)}. \end{equation}
Using that $\mu(I_{-})/\mu(I) = 1 - \mu(I_{+})/\mu(I)$ (and three other similar formulae), it is easy to see that the numbers on the right hand sides of \eqref{form4} and \eqref{form5} agree. So, $a_{I}$ can be defined consistently, and \eqref{form1} holds for $I_{+},I_{-} \in \calT$. Moreover, the formulae for $a_{I}$ look quite familiar:
\begin{obs}\label{mainObs} $|a_{I}| = \Delta_{\mu,\nu}(I)$ for $I \in \calT \setminus \Leaves(\calT)$.
\end{obs}
Now that the coefficients $a_{I}$ have been successfully defined for $I \in \calT \setminus \Leaves(\calT)$, let $g$ be the (at the moment) formal series
\begin{displaymath} g(x) := \sum_{I \in \calT \setminus \Leaves(\calT)} a_{I}h_{I}^{\mu}(x). \end{displaymath}
Since the Haar functions $h_{I}^{\mu}$ are orthogonal in $L^{2}(\mu)$, and satisfy
\begin{displaymath} \int (h_{I}^{\mu})^{2} \, d\mu \leq \max\{c_{I}^{+},c_{I}^{-}\}^{2}\mu(I) \leq D^{2}\mu(I), \quad I \in \calT \setminus \Leaves(\calT), \end{displaymath}
one arrives at
\begin{displaymath} \|g\|_{L^{2}(\mu)}^{2} = \sum_{I \in \calT \setminus \Leaves(\calT)} \Delta_{\mu,\nu}^{2}(I)\|h_{I}\|_{L^{2}(\mu)}^{2} \leq D^{2} \sum_{I \in \calT \setminus \Leaves(\calT)} \Delta_{\mu,\nu}^{2}(I)\mu(I) < \infty, \end{displaymath}
by the assumption in Proposition \ref{mainProp}. This means that the sequence
\begin{displaymath} g_{N} := \mathop{\sum_{I \in \calT \setminus \Leaves(\calT)}}_{|I| > 2^{-N}} a_{I}h_{I}^{\mu} \end{displaymath}
converges in $L^{2}(\mu)$. In particular, one can pick a subsequence $(g_{N_{j}})_{j \in \N}$, which converges pointwise $\mu$ almost everywhere (in fact, the entire sequence converges by basic martingale theory, but this is not needed). Now, recall that the goal was to prove that $\mu|_{\Top(\calT)} \ll \nu_{\calT}$. To this end, one has to verify that
\begin{equation}\label{form22} \liminf_{I \to x} \tfrac{\mu}{\nu_{\calT}}(I) < \infty \end{equation}
at $\mu$ almost every $x \in \Top(\calT)$. This is clear for $x \in \Leaves(\calT)$, since the ratios $\mu(I)/\nu_{\calT}(I)$, $I \ni x$, are eventually constant. So, it suffices to prove \eqref{form22} at $\mu$ almost every point $x \in \partial \calT$. Fix a point $x \in \partial \calT$ with the properties that sequence $(g_{N_{j}}(x))_{j \in \N}$ converges, and also
\begin{equation}\label{form23} \sum_{x \in J \in \calT} a_{J}^{2} = \sum_{x \in J \in \calT} \Delta_{\mu,\nu}^{2}(I) < \infty. \end{equation}
These properties hold at $\mu$ almost every $x \in \partial \calT$. Let $I \in \calD$ be so small that $x \in I \in \calT$, and note that
\begin{displaymath} \log \tfrac{\nu_{\calT}}{\mu}(I) = \log \tfrac{\nu}{\mu}(I) = \log \mathop{\prod_{J \supsetneq I}}_{J \in \calT} (1 + a_{J}h_{J}^{\mu}(x)) = \mathop{\sum_{J \supsetneq I}}_{J \in \calT} \log (1 + a_{J}h_{J}^{\mu}(x)). \end{displaymath}
Now, the plan is to use the estimate $\log(1 + t) \geq t - C_{\delta}t^{2}$, valid as long as $t \geq \delta - 1$ for some $\delta > 0$. Observe that $a_{J}h_{J}^{\mu}(x) \in \{-c_{J}^{-}a_{J},c_{J}^{+}a_{J}\}$, where
\begin{equation}\label{form39} -a_{J}c_{J}^{-} = \frac{\tfrac{\nu}{\mu}(J_{-})}{\tfrac{\nu}{\mu}(J)} - 1 \geq \frac{1}{C} - 1 \quad \text{and} \quad a_{J}c_{J}^{+} = \frac{\tfrac{\nu}{\mu}(J_{+})}{\tfrac{\nu}{\mu}(J)} - 1 \geq \frac{1}{C} - 1. \end{equation}
Consequently, for $x \in I \in \calT$ with $|I| = 2^{-N_{j}}$, one has
\begin{equation}\label{form40} \log \tfrac{\nu_{\calT}}{\mu}(I) \geq \mathop{\sum_{J \supsetneq I}}_{J \in \calT} a_{J}h_{J}^{\mu}(x) -  C'\mathop{\sum_{J \supsetneq I}}_{J \in \calT} (a_{J}h_{J}^{\mu}(x))^{2} \geq g_{N_{j}}(x) - C'D^{2}\sum_{x \in J \in \calT} a_{J}^{2}, \end{equation}
where $C' \lesssim_{D} 1$ only depends on the constant $C$ in \eqref{form39}. Since the sequence $(g_{N_{j}}(x))_{j \in \N}$ converges and \eqref{form23} holds, the right hand side of \eqref{form40} has a uniform lower bound $-M(x) > -\infty$. This implies that
\begin{displaymath} \limsup_{I \to x} \tfrac{\nu_{\calT}}{\mu}(I) \geq \exp(-M(x)) > 0, \end{displaymath}
which gives \eqref{form22} at $x$. The proof of Proposition \ref{mainProp} is complete.

\section{Proof of Theorem \ref{main}(b)}\label{mainProof} In this section, Theorem \ref{main}(b) is proved via a simple tree construction, coupled with Propositions \ref{deltaVsAlpha} and \ref{mainProp}. Recall the statement of Theorem \ref{main}(b):
\begin{thm}\label{main2} Assume that $\mu,\nu$ are Borel probability measures on $[0,1)$, $\mu$ does not charge the boundaries of dyadic intervals, and $\nu$ is dyadically doubling. Write $\mu = \mu_{a} + \mu_{s}$ for the Lebesgue decomposition of $\mu$ relative to $\nu$, and let $\calS_{\calD,\nu}(\mu)$ for the square function
\begin{displaymath} \calS^{2}_{\calD,\nu}(\mu) = \sum_{I \in \calD} \alpha_{\mu,\nu}^{2}(I)\chi_{I}. \end{displaymath}
Then, $\calS_{\calD,\nu}(\mu)$ is infinite $\mu_{s}$ almost surely.
\end{thm}
An equivalent statement is that the restriction of $\mu$ to the set
\begin{displaymath} G := \{x \in [0,1) : \calS_{\calD,\nu}(\mu)(x) < \infty\} \end{displaymath}
is absolutely continuous with respect to $\nu$; this is the formulation proven below. For the rest of the section, fix the measures $\mu,\nu$ as in the statement above, and let $D$ be the doubling constant of $\nu$. I record a simple lemma, which says that the doubling of $\nu$ implies the doubling of $\mu$ on intervals, where the $\alpha$-number is small enough.
 \begin{lemma}\label{decayLemma} There are constants $\epsilon > 0$ and $C \geq 1$, depending only on $D$, such that the following holds. For every interval $I \in \calD$, if $\alpha_{\mu,\nu}(I) < \epsilon$, then
\begin{equation}\label{decay} \mu(I) \leq C\min\{\mu(I_{-}),\mu(I_{+})\}. \end{equation} 
\end{lemma}

\begin{proof} Let $I_{--} \subset I_{-}$ and $I_{++} \subset I_{+}$ be intervals, which lie at distance $\geq |I|/8$ from the boundaries of $I_{-}$ and $I_{+}$, respectively, and have length $|I|/8$. Let $\psi_{-}$ and $\psi_{+} \colon \R \to [0,1]$ be $(C'/|I|)$-Lipschitz functions, which equal $1$ on $I_{--}$ and $I_{++}$, respectively, and are supported on $I_{-}$ and $I_{+}$. Then
\begin{displaymath} \frac{\mu(I_{-})}{\mu(I)} \geq \frac{1}{\mu(I)}\int \psi_{-} \, d\mu \geq \frac{1}{\nu(I)}\int \psi_{-} \, d\nu - C'\alpha_{\mu,\nu}(I) \geq \frac{\nu(I_{--})}{\nu(I)} - C'\alpha_{\mu,\nu}(I), \end{displaymath}
and the analogous inequality holds for $\mu(I_{+})/\mu(I)$. The ratio $\nu(I_{--})/\nu(I)$ is at least $1/D^{3}$, so if $\alpha_{\mu,\nu}(I) < 1/(2C'D^{3}) =: \epsilon$, then both $\mu(I_{-}) \geq [1/(2D^{3})]\mu(I)$ and $\mu(I_{+}) \geq [1/(2D^{3})]\mu(I)$. This gives \eqref{decay} with $C = 2D^{3}$. \end{proof}

In particular, if $\calT$ is a tree, and $\alpha_{\mu,\nu}(I) < \epsilon$ for all $I \in \calT \setminus \Leaves(\calT)$, then $\mu$ is $(\calT,C)$-doubling. I will now describe, how such trees $\calT_{j} \subset \calD$ are constructed, starting with $\calT_{0}$. Let $[0,1) = \Top(\calT_{0})$, and assume that some interval $I \in \calT_{0}$. If
\begin{equation}\label{form25} \sum_{I \subset J \subset [0,1)} \alpha_{\mu,\nu}^{2}(J) \geq \epsilon^{2}, \end{equation}
add $I$ to $\Leaves(\calT_{0})$. The children $I_{-}$ and $I_{+}$ become the tops of new trees. If \eqref{form25} fails, add $I_{-}$ and $I_{+}$ to $\calT_{0}$. The construction of $\calT_{0}$ is now complete. If a new top $T_{j}$ was created in the process of constructing $\calT_{0}$, and $\mu(T_{j}) > 0$, construct a new tree $\calT_{j}$ with $\Top(\calT_{j}) = T_{j}$ by repeating the algorithm above, only replacing $[0,1)$ by $T_{j}$ in the stopping criterion \eqref{form25}. Continue this process until all intervals in $\calD$ belong to some tree, or all remaining tops $T_{j}$ satisfy $\mu(T_{j}) = 0$. For all tops $T_{j}$ with $\mu(T_{j}) = 0$, simply define $\calT_{j} := \{I \in \calD : I \subset T_{j}\}$, so there is no further stopping inside $\calT_{j}$.

\begin{remark} Let $\calT$ be one of the trees constructed above, with $\mu(\Top(\calT)) > 0$. Then $\mu$ is $(\calT,C)$-doubling by Lemma \ref{decayLemma}, since it is clear that $\alpha_{\mu,\nu}(I) < \epsilon$ for all $I \in \calT \setminus \Leaves(\calT)$. In particular $\mu(I) > 0$ for all $I \in \calT$. \end{remark}

The following observation is now rather immediate from the definitions:

\begin{lemma}\label{topLemma} Assume that $\calT_{0},\ldots,\calT_{N - 1}$ are distinct trees such that $x \in \Leaves(\calT_{j})$ for all $0 \leq j \leq N - 1$. Then
\begin{displaymath} \calS^{2}_{\calD,\nu}(\mu)(x) \geq \epsilon^{2} N. \end{displaymath}
\end{lemma}

\begin{proof} For $0 \leq j \leq N - 1$, Let $I_{j} \in \Leaves(\calT_{j})$ with $x \in I_{j}$. Then
\begin{displaymath} \calS^{2}_{\calD,\nu}(\mu)(x) \geq \sum_{j = 0}^{N - 1} \sum_{I_{j} \subset J \subset \Top(\calT_{j})} \alpha_{\mu,\nu}^{2}(J) \geq \epsilon^{2} N, \end{displaymath}
as claimed. \end{proof}

It follows that $\mu$ almost every point in $G = \{x \in [0,1) : \calS_{\nu}(\mu)(x) < \infty\}$ belongs to $\Leaves(\calT_{j})$ for only finitely many trees $\calT_{j}$. This is equivalent to saying that $\mu$ almost every point in $G$ belongs to $\partial \calT$ for some tree $\calT$. The converse is also true: if $x$ belongs to $\partial \calT$ for some tree $\calT$, then clearly $\calS_{\nu}(\mu)(x) < \infty$. Consequently
\begin{displaymath} \mu|_{G} = \sum_{\text{trees }\calT} \mu|_{\partial \calT}. \end{displaymath}
To prove  Theorem \ref{main2}, it now suffices to show that $\mu|_{\partial \calT} \ll \nu$ for every tree $\calT$. This is clear, if $\mu(\Top(\calT)) = 0$, so I exclude the trivial case to begin with. In the opposite case, note that
\begin{equation}\label{form28} \sum_{I \in \calT \setminus \Leaves(\calT)} \alpha_{\mu,\nu}^{2}(I)\mu(I) = \int \sum_{I \in \calT \setminus \Leaves(\calT)} \alpha_{\mu,\nu}^{2}(I)\chi_{I}(x) \, d\mu x \leq \epsilon^{2} \cdot \mu(\Top(\calT)). \end{equation}
It then follows from Proposition \ref{deltaVsAlpha} that
\begin{displaymath} \sum_{I \in \calT} \Delta^{2}_{\mu,\nu}(I)\mu(I) \lesssim \mu(\Top(\calT)) < \infty, \end{displaymath}
and the claim $\mu|_{\partial \calT} \ll \nu$ is finally a consequence of Proposition \ref{mainProp}. The proof of Theorem \ref{main}(b) is complete.

\section{The non-dyadic square function}\label{continuousVariants}

This section contains the proof of Theorem \ref{mainCont}(b). The argument naturally contains many similarities to the one given above. The main novelty is that one needs to work with the smooth $\alpha$-numbers, introduced in Definition \ref{smoothAlphas} (or \cite[Section 5]{ADT}). 

\subsection{Smooth $\alpha$-numbers, and their properties} I recall the definition of the smooth $\alpha$-numbers:
\begin{definition}[Smooth $\alpha$-numbers] Write $\varphi(x) = \dist(x,\R \setminus (0,1))$. For an interval $I \subset \R$, define $\alpha_{s,\mu,\nu}(I) := \W_{1}(\mu_{\varphi,I},\nu_{\varphi,I})$, where 
\begin{displaymath} \mu_{\varphi,I} := \frac{T_{I\sharp}(\mu|_{I})}{\mu(\varphi_{I})} \quad \text{and} \quad  \nu_{\varphi,I} := \frac{T_{I\sharp}(\nu|_{I})}{\nu(\varphi_{I})}. \end{displaymath}
Here $\varphi_{I} = \varphi \circ T_{I}$, and $\mu(\varphi_{I}) = \int \varphi_{I} \, d\mu$. If $\mu(\varphi_{I}) = 0$ (or $\nu(\varphi_{I}) = 0$), set $\mu_{\varphi,I} \equiv 0$ (or $\nu_{\varphi,I} \equiv 0$). Unwrapping the definition, if $\mu(\varphi_{I}),\nu(\varphi_{I}) > 0$, then
\begin{displaymath} \alpha_{s,\mu,\nu}(I) = \sup_{\psi} \left| \frac{1}{\mu(\varphi_{I})} \int \psi \circ T_{I} \, d\mu - \frac{1}{\nu(\varphi_{I})} \int \psi \circ T_{I} \, d\nu \right| = \sup_{\psi} \left| \frac{\mu(\psi_{I})}{\mu(\phi_{I})} - \frac{\nu(\psi_{I})}{\nu(\phi_{I})} \right|, \end{displaymath}
where the $\sup$ is taken over test functions $\psi$.
\end{definition}

Recall that the main reason to prefer the smooth $\alpha$-numbers over the ones from Definition \ref{alphas} is the following stability property: if $I \subset J$ are intervals of comparable length, then $\alpha_{s,\mu,\nu}(I) \lesssim \alpha_{s,\mu,\nu}(J)$, whenever either $\mu$ or $\nu$ is doubling. This fact is essentially \cite[Lemma 5.2]{ADT2}, but I include a proof in Proposition \ref{basicProperties} for completeness. Similar stability is not true for the numbers $\alpha_{\mu,\nu}(I)$ and $\alpha_{\mu,\nu}(J)$, even for very nice measures $\mu$ and $\nu$, as the following example demonstrates:
\begin{ex}\label{counterEx} Fix $n \in \N$, and let $I^{n}_{-} := [\tfrac{1}{2} - 2^{-n},\tfrac{1}{2}]$ and $I^{n}_{+} := (\tfrac{1}{2},\tfrac{1}{2} + 2^{-n}]$. Let $\mu$ be the same measure as in Example \ref{exampleOne}:
\begin{displaymath} \mu = \chi_{\R \setminus (I^{n}_{-} \cup I^{n}_{+})} + \frac{\chi_{I^{n}_{-}}}{2} + \frac{3\chi_{I^{n}_{+}}}{2}. \end{displaymath}
Let $\nu = \calL$. It is clear that both $\mu$ and $\nu$ are doubling, with constants independent of $n$. It is also easy to check that $\alpha_{\mu,\nu}(I) \lesssim 2^{-2n}$ for any interval $I$ with length $|I| \sim 1$ such that $I_{-}^{n} \cup I_{+}^{n} \subset I$ (this implies that $\mu(I) = \nu(I)$). However, $\alpha_{\mu,\nu}([0,1/2]) \sim 2^{-n}$, because $\nu_{[0,1/2)} = \chi_{[0,1]}$, while 
\begin{displaymath} \mu_{[0,1/2]} = \left(1 + \frac{2^{-n}}{1 - 2^{-n}} \right) \chi_{[0,1 - 2^{1 - n})} + \frac{1}{2}\left(1 + \frac{2^{-n}}{1 - 2^{-n}} \right)\chi_{[1 - 2^{1 - n},1]}. \end{displaymath}
So, for instance, it is clear that no inequality of the form $\alpha_{\mu,\nu}([0,1/2]) \lesssim \alpha_{\mu,\nu}([-1,1])$ can hold. \end{ex}
Without any doubling assumptions, even the smooth $\alpha$-numbers can behave badly:
\begin{ex} Let $\mu = \delta_{1/2}$, and $\nu = (1 - \epsilon) \cdot \delta_{1/2 + \epsilon} + \epsilon \cdot \delta_{1/4}$. Then $\alpha_{s,\mu,\nu}([-1,1]) \sim \epsilon$, but $\alpha_{s,\mu,\nu}([0,1/2]) \sim 1$.\end{ex}

\begin{proposition}[Basic properties of the smooth $\alpha$-numbers]\label{basicProperties} Let $\mu,\nu$ be two Radon measures on $\R$, and let $I \subset \R$ be an interval. Then
\begin{displaymath} \alpha_{s,\mu,\nu}(I) \leq 2 \quad \text{and} \quad \alpha_{s,\mu,\nu}(I) \leq \frac{2\alpha_{\mu,\nu}(I)}{\nu_{I}(\varphi)}. \end{displaymath}
Moreover, if $\nu$ is doubling with constant $D$, the following holds. If $I \subset J \subset \R$ are intervals with $|I| \geq \theta|J|$ for some $\theta > 0$, then
\begin{equation}\label{comparability} \alpha_{s,\mu,\nu}(I) \lesssim_{D,\theta} \alpha_{s,\mu,\nu}(J). \end{equation}
\end{proposition}
\begin{proof} For the duration of the proof, fix an interval $I \subset \R$ with $\mu(\varphi_{I}),\nu(\varphi_{I}) > 0$. The cases, where $\mu(\varphi_{I}) = 0$ or $\nu(\varphi_{I}) = 0$ always require a little case chase, which I omit. Recall that $\varphi = \chi_{[0,1]}\dist(\cdot,\{0,1\})$. Note that any $1$-Lipschitz function $\psi \colon \R \to \R$ supported on $[0,1]$ must satisfy $|\psi| \leq \varphi$. Consequently $|\psi_{I}| \leq \varphi_{I}$ for any interval $I$, and so
\begin{displaymath} \alpha_{s,\mu,\nu}(I) \leq \sup_{\psi} \left[ \frac{\mu(|\psi_{I}|)}{\mu(\varphi_{I})} + \frac{\nu(|\psi_{I}|)}{\nu(\varphi_{I})} \right] \leq 2. \end{displaymath}
This proves the first inequality. For the second inequality, one may assume that $\alpha_{\mu,\nu}(I) > 0$, since otherwise $\mu|_{\operatorname{int} I} = c\nu|_{\operatorname{int} I}$ for some constant $c > 0$, and this also gives $\alpha_{s,\mu,\nu}(I) = 0$. After this observation, it is easy to reduce to the case $\mu(\varphi_{I}) > 0$ and $\nu(\varphi_{I}) > 0$. Fix a test function $\psi$. Using  that $\mu_{I}(|\psi|) = \mu(|\psi_{I}|)/\mu(I) \leq  \mu(\varphi_{I})/\mu(I) = \mu_{I}(\varphi)$, one obtains
\begin{align*} \left|\frac{\mu(\psi_{I})}{\mu(\varphi_{I})} - \frac{\nu(\psi_{I})}{\nu(\varphi_{I})} \right| & = \left|\frac{\mu_{I}(\psi)}{\mu_{I}(\varphi)} - \frac{\nu_{I}(\psi)}{\nu_{I}(\varphi)} \right| = \left|\frac{\mu_{I}(\psi)\nu_{I}(\varphi) - \nu_{I}(\psi)\mu_{I}(\varphi)}{\mu_{I}(\varphi)\nu_{I}(\varphi)} \right|\\
& \leq \frac{\mu_{I}(|\psi|)}{\mu_{I}(\varphi)\nu_{I}(\varphi)}|\mu_{I}(\varphi) - \nu_{I}(\varphi)| + \frac{\mu_{I}(\varphi)}{\mu_{I}(\varphi)\nu_{I}(\varphi)}|\mu_{I}(\psi) - \nu_{I}(\psi)| \leq \frac{2\alpha_{\mu,\nu}(I)}{\nu_{I}(\varphi)}. \end{align*} 

To prove the final claim, start with the following estimate for a test function $\psi$:
\begin{align*} \left| \frac{\mu(\psi_{I})}{\mu(\varphi_{I})} - \frac{\nu(\psi_{I})}{\nu(\varphi_{I})} \right| \leq \frac{\nu(\varphi_{J})}{\nu(\varphi_{I})}\left| \frac{\mu(\psi_{I})}{\mu(\varphi_{J})} - \frac{\nu(\psi_{I})}{\nu(\varphi_{J})} \right| + \frac{\mu(|\psi_{I}|)}{\mu(\varphi_{I})} \frac{\nu(\varphi_{J})}{\nu(\varphi_{I})} \left|\frac{\mu(\varphi_{I})}{\mu(\varphi_{J})} - \frac{\nu(\varphi_{I})}{\nu(\varphi_{J})} \right|. \end{align*} 
Then, recall that $\mu(|\psi_{I}|) \leq \mu(\varphi_{I})$. Further, it follows from the doubling of $\nu$ that $\nu(\varphi_{J}) \lesssim_{D,\theta} \nu(\varphi_{I})$. Finally, notice that $\psi_{I} = (\psi_{I} \circ T_{J}^{-1}) \circ T_{J}$ and $\varphi_{I} = (\varphi_{I} \circ T_{J}^{-1}) \circ T_{J}$, where both
\begin{displaymath} \psi_{I} \circ T_{J}^{-1} \quad \text{and} \quad \varphi_{I} \circ T_{J}^{-1} \end{displaymath}
are $(|J|/|I|)$-Lipschitz functions supported on $T_{J}(I) \subset [0,1]$. Consequently, 
\begin{displaymath} \max\left\{\left| \frac{\mu(\psi_{I})}{\mu(\varphi_{J})} - \frac{\nu(\psi_{I})}{\nu(\varphi_{J})} \right|, \left|\frac{\mu(\varphi_{I})}{\mu(\varphi_{J})} - \frac{\nu(\varphi_{I})}{\nu(\varphi_{J})} \right| \right\} \leq \frac{\alpha_{s,\mu,\nu}(J)}{\theta}, \end{displaymath}
and the estimate \eqref{comparability} follows. \end{proof}

\subsection{Proof of Theorem \ref{mainCont}(b)}\label{SmoothSquareFunction} In this section, $\nu$ is a globally doubling measure with constant $D \geq 1$, say. As in Section \ref{mainProof}, it suffices to show that $\mu|_{G} \ll \nu$, where
\begin{displaymath} G := \{x : \calS_{\nu}(\mu)(x) < \infty\}. \end{displaymath}
Write
\begin{displaymath} \alpha_{s,\mu,\nu}(J) =: \alpha(J), \qquad J \subset \R. \end{displaymath}
Assume without loss of generality (or translate both measures $\mu$ and $\nu$ slightly) that $\mu(\partial I) = 0$ for all $I \in \calD$. Also without loss of generality, one may assume that $\spt \mu \subset (0,1)$: the reason is that the finiteness $\calS_{\nu}(\mu)(x)$ is equivalent to the finiteness of $\calS_{\nu}(\mu|_{U})(x)$ for all $x \in U$, whenever $U \subset \R$ is open. So, it suffices to prove $\mu|_{U \cap G} \ll \nu$ for any bounded open set $U$. Whenever I write $\calD$ in the sequel, I only mean the family $\{I \in \calD : I \subset [0,1)\}$.

I start with some standard discretisation arguments. For each $I \in \calD$, associate a somewhat larger interval $B_{I} \supset I$ as follows. First, for $x \in \spt \mu$ and $k \in \N$, choose a radius $r_{x,k} > 0$ such that
\begin{equation}\label{rXK} \alpha(B(x,r_{x,k})) \leq 2 \inf \{\alpha(B(x,r)) : 1.1 \cdot 2^{-k - 1} \leq r \leq 0.9 \cdot 2^{-k}\}. \end{equation}
Then
\begin{displaymath} \alpha^{2}(B(x,r_{x,k})) \leq \left( \frac{1}{\ln [2 \cdot (0.9/1.1)]} \int_{1.1 \cdot 2^{-k - 1}}^{0.9 \cdot 2^{-k}} 2\alpha(x,r) \, \frac{dr}{r} \right)^{2} \lesssim \int_{2^{-k - 1}}^{2^{-k}} \alpha^{2}(x,r) \, \frac{dr}{r}. \end{displaymath}
For $I \in \calD$ with $|I| = 2^{-k}$ and $I \cap \spt \mu \neq \emptyset$, let $B_{I}$ be some open interval of the form $B(x,r_{k - 10})$, $x \in I$, such that
\begin{displaymath} \alpha(B_{I}) \leq 2 \inf \{\alpha(B(y,r_{y,k - 10})) : y \in I \cap \spt \mu\}. \end{displaymath}
The number "$-10$" simply ensures that $I \subset B_{I}$ with $\dist(I,\partial B_{I}) \sim |I|$, and
\begin{displaymath} I \subset J \quad \Longrightarrow \quad B_{I} \subset B_{J}, \qquad \text{for } I,J \in \calD. \end{displaymath}
This implication also uses the slight separation between the scales, provided by the factors "$1.1$" and "$0.9$" in \eqref{rXK}. For $I \in \calD$ with $I \cap \spt \mu = \emptyset$, define $B_{I} := I$ (although this definition will never be really used). Now, a tree decomposition of $\calD$ can be performed as in the previous section, replacing the stopping condition \eqref{form25} by declaring $\Leaves(\calT)$ to consist of the maximal intervals $I \subset \Top(\calT)$ with
\begin{displaymath} \sum_{I \subset J \subset \Top(\calT)} \alpha^{2}(B_{I}) \geq \epsilon^{2}, \end{displaymath}
where $\epsilon = \epsilon_{D} > 0$ is a suitable small number; in particular, $\epsilon > 0$ is chosen so small that $\alpha(B_{I}) \leq \epsilon$ implies $\mu(B_{I}) \lesssim \mu(I)$ (which is possible by a small modification of Lemma \ref{decayLemma}). If now $x \in \Leaves(\calT)$ for infinitely many different trees $\calT$, then 
\begin{displaymath} \infty = \sum_{x \in I \in \calD} \alpha^{2}(B_{I}) \leq 2\sum_{k \in \N} \alpha^{2}(B(x,r_{x,k - 10})) \lesssim \int_{0}^{2^{10}} \alpha^{2}(B(x,r)) \, \frac{dr}{r}, \end{displaymath}
which implies that $x \notin G$. Repeating the argument from Section \ref{mainProof}, this gives
\begin{displaymath} \mu|_{G} \leq \sum_{\text{trees }\calT} \mu|_{\partial \calT}. \end{displaymath}
The converse inequality could also be deduced from the stability of the smooth $\alpha$-numbers (Proposition \ref{basicProperties}), but it is not needed: the inequality already shows that it suffices to prove
\begin{equation}\label{TreeAbsCont} \mu|_{\partial \calT} \ll \nu \end{equation}
for any given tree $\calT$. So, fix a tree $\calT$. If $\epsilon > 0$ was chosen small enough (again depending on $D$), then $\mu$ is $(\calT,C)$-doubling for some $C = C_{D} \geq 1$ in the usual sense:
\begin{displaymath} \mu(\hat{I}) \leq C\mu(I), \qquad I \in \calT \setminus \Top(\calT). \end{displaymath}
So, if one knew that
\begin{equation}\label{form27} \sum_{I \in \calT \setminus \Leaves(\calT)} \Delta_{\mu,\nu}^{2}(I)\mu(I) < \infty, \end{equation}
then the familiar Proposition \ref{mainProp} would imply \eqref{TreeAbsCont}, completing the entire proof.

The proof of \eqref{form27} is based on the following inequality:
\begin{equation}\label{form29} \sum_{I \in \calT} \Delta_{\mu,\nu}^{2}(I)\mu(I) \lesssim \sum_{I \in \calT \setminus \Leaves(\calT)} \alpha^{2}(B_{I})\mu(B_{I}) + \mu(\Top(\calT)). \end{equation}
The right hand side is finite by the same estimate as in \eqref{form28} (start with $\mu(B_{I}) \lesssim \mu(I)$, using $\alpha(B_{I}) \leq \epsilon$ for $I \in \calT \setminus \Leaves(\calT)$). So, \eqref{form29} implies \eqref{form27}. I start the proof of \eqref{form29} by noting that if $I \in \calD$, then
\begin{align}\label{comparisonIneq} \Delta_{\mu,\nu}(I) & = \left|\frac{\nu(I_{-})}{\nu(I)} - \frac{\mu(I_{-})}{\mu(I)}\right|\\
& \leq \frac{\nu(\varphi_{B_{I}})}{\nu(I)} \left|\frac{\nu(I_{-})}{\nu(\varphi_{B_{I}})} - \frac{\mu(I_{-})}{\mu(\varphi_{B_{I}})}\right|  + \frac{\mu(I_{-})}{\mu(I)}\frac{\nu(\varphi_{B_{I}})}{\nu(I)} \left|\frac{\mu(I)}{\mu(\varphi_{B_{I}})} - \frac{\nu(I)}{\nu(\varphi_{B_{I}})} \right|. \notag \end{align} 
Noting that $\nu(\varphi_{B_{I}})/\nu(I) \lesssim_{D} 1$, to prove \eqref{form29}, it suffices to control 
\begin{equation}\label{form30} \sum_{I \in \calT \setminus \Leaves(\calT)} \left[ \left|\frac{\nu(I_{-})}{\nu(\varphi_{B_{I}})} - \frac{\mu(I_{-})}{\mu(\varphi_{B_{I}})}\right|^{2} + \left|\frac{\mu(I)}{\mu(\varphi_{B_{I}})} - \frac{\nu(I)}{\nu(\varphi_{B_{I}})} \right|^{2}\right]\mu(I) \end{equation}
by the right hand side of \eqref{form29}. The main task it to find a suitable replacement for the "$\Tail-\Tip$" inequality \eqref{CalF+}, which I replicate here for comparison:
\begin{equation}\label{tipTail} \Delta_{\mu,\nu}(I)\mu(I) \leq \frac{C}{\tau} \sum_{J \in \Tail_{I}} \alpha_{\mu,\nu}(J)\mu(J) + \kappa \sum_{J \in \Tail_{I}} \Delta_{\mu,\nu}(J)\mu(J) + 2\mu(\Tip_{I}). \end{equation}
Glancing at \eqref{form30}, one sees that an analogue for the inequality above is actually needed for both the terms 
\begin{displaymath} \tilde{\Delta}_{B_{I}}(I_{-}) = \left|\frac{\nu(I_{-})}{\nu(\varphi_{B_{I}})} - \frac{\mu(I_{-})}{\mu(\varphi_{B_{I}})}\right| \quad \text{and} \quad \tilde{\Delta}_{B_{I}}(I) = \left|\frac{\mu(I)}{\mu(\varphi_{B_{I}})} - \frac{\nu(I)}{\nu(\varphi_{B_{I}})} \right|. \end{displaymath}
If $I_{-} \in \Leaves(\calT)$, then the trivial estimate $\tilde{\Delta}_{B_{I}}(I_{-}) \lesssim 1$ will suffice, so in the sequel I assume that 
\begin{equation}\label{excludedCase} I,I_{-} \notin \Leaves(\calT). \end{equation}
The goal is inequality \eqref{realTipTail} below. Fix $B_{I}$ and $J \in \{I,I_{-}\}$. Assume for notational convenience that $|B_{I}| = 1$, and hence, also $|J| \sim 1$. In a familiar manner, start by writing
\begin{equation}\label{chiJ} \chi_{J} = \sum_{k \in \Z} \psi_{k}, \end{equation}
where $\psi_{k}$ is a non-negative $C2^{|k|}$-Lipschitz function supported on either $J \subset B_{I}$ (for $k = 0$), or $J_{|k|-}$ (for negative $k$) or $J_{k+}$ (for positive $k$). As in the proof of the original $\Tail-\Tip$ inequality, it suffices to first estimate
\begin{equation}\label{form31} \left|\frac{1}{\mu(\varphi_{B_{I}})} \int \Psi^{+}_{0} \, d\mu - \frac{1}{\nu(\varphi_{B_{I}})} \int \Psi^{+}_{0} \, d\nu \right|, \end{equation}
where $\Psi^{+}_{0} = \sum_{k \geq 1} \psi_{k} + \psi_{0}/2$, and more generally $\Psi^{+}_{j} = \sum_{k \geq j} \psi_{j}$ for $j \geq 1$; eventually one can just replicate the argument for the function $\Psi_{0}^{-} = \sum_{k \leq -1} \psi_{k} + \psi_{0}/2$, and summing the bounds gives control for $\tilde{\Delta}_{B_{I}}(J)$. Start with the following estimate, which only uses the triangle inequality, and the fact that $\psi_{0}/2$ is a $C$-Lipschitz function supported on $B_{I}$:
\begin{align} &\left| \frac{1}{\mu(\varphi_{B_{I}})} \int \Psi^{+}_{0} \, d\mu - \frac{1}{\nu(\varphi_{B_{I}})} \int \Psi^{+}_{0} \, d\nu \right| \leq C\alpha(B_{I}) \notag \\
& \quad + \frac{\mu(\varphi_{B_{J_{+}}})}{\mu(\varphi_{B_{I}})} \left| \frac{1}{\mu(\varphi_{B_{J_{+}}})} \int \Psi^{+}_{1} \, d\mu - \frac{1}{\nu(\varphi_{B_{J_{+}}})} \int \Psi^{+}_{1} \, d\nu \right| \notag \\
&\label{form32} \quad + \left(\frac{1}{\nu(\varphi_{B_{J_{+}}})} \int \Psi_{1}^{+} \, d\nu \right) \left| \frac{\mu(\varphi_{B_{J_{+}}})}{\mu(\varphi_{B_{I}})} - \frac{\nu(\varphi_{B_{J_{+}}})}{\nu(\varphi_{B_{I}})} \right|. \end{align}
Here
\begin{displaymath} \frac{1}{\nu(\varphi_{B_{J_{+}}})} \int \Psi_{1}^{+} \, d\nu \lesssim 1, \end{displaymath}
since $\nu$ is doubling and $\Psi_{1}^{+}$ vanishes outside $J_{+} \subset B_{J_{+}}$, and
\begin{displaymath} \left| \frac{\mu(\varphi_{B_{J_{+}}})}{\mu(\varphi_{B_{I}})} - \frac{\nu(\varphi_{B_{J_{+}}})}{\nu(\varphi_{B_{I}})} \right| \leq \frac{|B_{I}|}{|B_{J_{+}}|} \cdot \alpha(B_{I}) \lesssim \alpha(B_{I}), \end{displaymath} 
since $\varphi_{B_{J_{+}}} = (\varphi_{B_{J_{+}}} \circ T_{B_{I}}^{-1}) \circ T_{B_{I}}$, where $\varphi_{B_{J_{+}}} \circ T_{B_{I}}^{-1}$ is a $(|B_{I}|/|B_{J_{+}}|)$-Lipschitz function supported on $[0,1]$. Consequently,
\begin{align*} & \left| \frac{1}{\mu(\varphi_{B_{I}})} \int \Psi^{+}_{0} \, d\mu - \frac{1}{\nu(\varphi_{B_{I}})} \int \Psi^{+}_{0} \, d\nu \right|\mu(\varphi_{B_{I}}) \leq C\alpha(B_{I})\mu(\varphi_{B_{I}})\\
& \quad + \left| \frac{1}{\mu(\varphi_{B_{J_{+}}})} \int \Psi^{+}_{1} \, d\mu - \frac{1}{\nu(\varphi_{B_{J_{+}}})} \int \Psi^{+}_{1} \, d\nu \right|\mu(\varphi_{B_{J_{+}}}) \end{align*}
Here $\Psi_{1}^{+}$ vanishes outside on $J_{+} \subset B_{J_{+}}$, so the estimate can be iterated. After $N \geq 0$ repetitions (the case $N = 0$ was seen above), one ends up with
\begin{align} & \left| \frac{1}{\mu(\varphi_{B_{I}})} \int \Psi^{+}_{0} \, d\mu - \frac{1}{\nu(\varphi_{B_{I}})} \int \Psi^{+}_{0} \, d\nu \right|\mu(\varphi_{B_{I}}) \leq C\sum_{k = 0}^{N} \alpha(B_{J_{k+}})\mu(\varphi_{B_{J_{k+}}})\notag \\
&\label{form33}\quad  + \mu(\varphi_{B_{J_{(N + 1)+}}})\left|\frac{1}{\mu(\varphi_{B_{(N + 1)+}})} \int \Psi^{+}_{N + 1} \, d\mu - \frac{1}{\nu(B_{(N + 1)+})} \int \Psi^{+}_{N + 1} \, d\nu \right|,  \end{align} 
where one needs to intepret $J_{0+} = I$ (which is different from $J$ in case $J = I_{-}$). What is a good choice for $N$? Let $N_{1} \geq 0$ be the smallest number such that $J_{(N_{1} + 1)+} \in \Leaves(\calT)$. If there is no such number, let $N_{1} = \infty$. In case $N_{1} = \infty$, the term on line \eqref{form33} vanishes, since $\mu(B_{J_{N+}})$ decays rapidly as long as $N \in \calT$ (using the doubling of $\nu$, and the fact that $\alpha(B_{I}) \leq \epsilon$ for $I \in \calT$). If $N_{1} < \infty$, the term on line \eqref{form33} is clearly bounded by $\leq 2\mu(B_{J_{(N_{1} + 1)+}})$, since $\Psi^{+}_{N_{1} + 1}$ vanishes outside $J_{(N_{1} + 1)+}$, which is well inside $B_{(N_{1} + 1)+}$. Observing that also $\mu(I) \lesssim \mu(\varphi_{B_{I}})$, it follows that
\begin{displaymath} \left| \frac{1}{\mu(\varphi_{B_{I}})} \int \Psi^{+}_{0} \, d\mu - \frac{1}{\nu(\varphi_{B_{I}})} \int \Psi^{+}_{0} \, d\nu \right|\mu(I) \lesssim \sum_{k = 0}^{N_{1}} \alpha(B_{J_{k+}})\mu(B_{J_{k+}}) + \mu(B_{J_{(N_{1} + 1)+}}). \end{displaymath}
Finally, by symmetry, the same argument can be carried out for the series $\Psi_{0}^{-} = \sum_{k < 0} \psi_{k} + \psi_{0}/2$. If $N_{2} \geq 0$ is the smallest number such that $J_{(N_{2} + 1)-} \in \Leaves(\calT)$, this leads to the following analogue of the $\Tail-\Tip$ inequality:
\begin{equation}\label{realTipTail} \tilde{\Delta}_{B_{I}}(J)\mu(I) \lesssim \sum_{P \in \Tail_{J}} \alpha(B_{P})\mu(B_{P}) + \mu(\Tip_{J}), \quad J \in \{I,I_{-}\}, I \in \calT \setminus \Leaves(\calT). \end{equation}
Here $\Tail_{J}$ is the collection of dyadic intervals $\Tail_{J} = \{J_{N_{2}-},\ldots,J,\ldots,J_{N_{1}+}\} \subset \calT \setminus \Leaves(\calT)$, and $\Tip_{J} = B_{J_{(N_{2} + 1)-}} \cup B_{J_{(N_{1} + 1)+}}$. Finally, in the excluded special case, where $J = I_{-} \in \Leaves(\calT)$ (recall \eqref{excludedCase}), the same estimate holds, if one defines $\Tail_{J} = \emptyset$ and $\Tip_{J} := J$ (noting that $I \in \calT$, so $\mu(I) \lesssim \mu(J)$).

Armed with the $\Tail-\Tip$ inequality \eqref{realTipTail}, the proof of the main estimate \eqref{form29} is a replica of the argument in the dyadic case, namely the proof of Proposition \ref{deltaVsAlpha}. I only sketch the details. For $I \in \calT \setminus \Leaves(\calT)$, and $J \in \{I,I_{-}\}$, start with 
\begin{align*} \tilde{\Delta}_{B_{I}}^{2}(J)\mu(I) & \lesssim \sum_{P \in \Tail_{J}} \alpha^{2}(B_{P}) \frac{\mu(B_{P})^{3/2}}{\mu(I)^{1/2}} + \frac{\mu(\Tip_{J})^{2}}{\mu(I)}\\
& \leq \mathop{\sum_{P \in \calT \setminus \Leaves(\calT)}}_{P \subset I} \alpha^{2}(B_{P}) \frac{\mu(B_{P})^{3/2}}{\mu(I)^{1/2}} + \frac{\mu(\Tip_{J})^{2}}{\mu(I)}. \end{align*}
The second inequality is trivial, and the first is proved with the same Cauchy-Schwarz argument as \eqref{form20}, using the fact that that $\sum_{P \in \Tail_{J}} \mu(B_{P})^{1/2} \lesssim \mu(I)^{1/2}$, which follows from $\Tail_{J} \subset \calT \setminus \Leaves(\calT)$, and in particular the geometric decay of the measures $\mu(B_{P})$ for $P \in \calT \setminus \Leaves(\calT)$. Now, the inequality above can be summed for $I \in \calT \setminus \Leaves(\calT)$ precisely as in the proof of \eqref{form21}. In particular, one should first use the estimate
\begin{displaymath} \mu(\Tip_{J}) \leq \mu(B_{J_{(N_{2} + 1)-}}) + \mu(B_{J_{(N_{1} + 1)+}}) \lesssim \mu(J_{(N_{2} + 1)-}) + \mu(J_{(N_{1} + 1)+}), \end{displaymath}
which follows from $\alpha(B_{J_{N_{1}+}}),\alpha(B_{J_{N_{2}-}}) < \epsilon$, if $\epsilon$ is small enough, depending on the doubling constant of $\nu$. The conclusion is
\begin{displaymath} \sum_{I \in \calT \setminus \Leaves(\calT)} \tilde{\Delta}_{B_{I}}^{2}(J)\mu(I) \lesssim \sum_{P \in \calT \setminus \Leaves(\calT)} \alpha^{2}(B_{P})\mu(B_{P}) + \mu(\Leaves(\calT)) \end{displaymath}
for $J \in \{I,I_{-}\}$. As observed in and around \eqref{form30}, this implies \eqref{form29}. 

\begin{remark}\label{AinftyRemark} In the proof of \eqref{form29}, the uniform bound $\alpha(B_{I}) < \epsilon$, $I \in \calT \setminus \Leaves(\calT)$, was only used to guarantee that $\mu$ is sufficiently doubling along, and inside, the balls $B_{I}$. If such properties are assumed \emph{a priori} in some given tree $\calT$, then \eqref{form29} continues to hold for $\calT$. In particular, if $\mu$ is doubling on the whole real line, and Carleson condition 
\begin{displaymath} \int_{B(x,2r)} \int_{0}^{2r} \alpha^{2}_{\mu,\nu}(B(y,t)) \, \frac{dt \, d\mu y}{t} \leq C\mu(B(x,r)), \end{displaymath}
holds, then the dyadic Carleson condition of Theorem \ref{mainCarleson} holds for any dyadic system $\calD$ (a family of half-open intervals covering $\R$, where every interval has length of the form $2^{-k}$ for some $k \in \Z$, and every interval is the union of two further intervals in the family; the proof of Theorem \ref{mainCarleson} seen in Section \ref{deltaAlphaComparison} works for any such system). It follows from this that $\mu \in A^{\calD}_{\infty}(\nu)$ for every dyadic system $\calD$, and consequently $\mu \in A_{\infty}(\nu)$. (To see this, pick a finite collection $\calD_{1},\ldots,\calD_{N}$ of dyadic systems so that the $\max$ of the corresponding dyadic maximal functions $M_{\nu}^{\calD_{i}}$,
\begin{displaymath} M_{\nu}^{\calD_{i}}f(x) = \sup_{x \in I \in \calD_{i}} \frac{1}{\nu(I)} \int_{I} |f| \, d\nu, \end{displaymath}
bounds the usual Hardy-Littlewood maximal function $M_{\nu}$, up to a constant depending only on the doubling of $\nu$. The construction of such systems is well-known, and in $\R$ as few as $2$ systems do the trick; for a reference, see for instance Section 5 in \cite{MTT}. Then, for every $1 \leq i \leq N$, there exists $p_{i} < \infty$ such that $\mu \in A_{p_{i}}^{\calD_{i}}(\nu)$, see \cite[Theorem 9.33(f)]{Gr}. In particular $\mu \in A_{p}^{\calD_{i}}(\nu)$ for $p := \max p_{i}$, and hence $\|M_{\nu}^{\calD_{i}}\|_{L^{p}(\mu) \to L^{p}(\mu)} < \infty$ for $1 \leq i \leq N$. It follows that $\|M_{\nu}\|_{L^{p}(\mu) \to L^{p}(\mu)} < \infty$, which is one possible definition for $\mu \in A_{\infty}(\nu)$. For much more information, see \cite[Section 9.11]{Gr}.) This proves the "continuous" part of Theorem \ref{mainCarleson}.
\end{remark}

\section{Parts (a) of the main theorems}\label{muASection}

Parts (a) of Theorems \ref{main} and \ref{mainCont} are proved in this section: $\mathcal{S}_{\calD,\nu}(\mu)$ and $\mathcal{S}_{\nu}(\mu)$ are finite $\mu_{a}$ almost everywhere, where $\mu_{a}$ is the absolutely continuous part of $\mu$ relative to $\nu$. The strategy is to prove the statement first for the dyadic square function $\mathcal{S}_{\calD,\nu}(\mu)$, but allow $\calD$ to be a slightly generalised system: a family $\calD = \cup \calD_{k}$, $k \geq 0$, of half-open intervals of length at most one such that
\begin{itemize}
\item[(D1)] each $\calD_{k}$ is a partition of $\R$,
\item[(D2)] each interval in $\calD_{k}$ has length $2^{-k}$, and 
\item[(D3)] each interval $I \in \calD_{k}$ has two children in $\calD_{k + 1}$, denoted by $\ch(I)$.
\end{itemize}
The added generality makes no difference in the proof, which closely follows previous arguments of Tolsa from \cite{To2} and \cite{To3}. The benefit is that the non-dyadic square function $\mathcal{S}_{\nu}(\mu)$ can, eventually, be bounded by a finite sum of dyadic square functions $\calS_{\calD_{1},\nu}(\mu),\ldots,\calS_{\calD_{N},\nu}(\mu)$, so the non-dyadic problem easily reduces to the dyadic one.

With the strategy in mind, fix a dyadic system $\calD$ satisfying (D1)-(D3), and let $\calS_{\calD,\nu}(\mu)$ be the associated square function. 

\begin{lemma}\label{lemma2} Assume that $\mu,\nu$ are Radon measures on $\R$, with $\mu$ finite, and $\nu$ dyadically doubling (relative to $\calD$). Then $\calS_{\nu}(\mu)$ is finite $\mu_{a}$ almost surely. \end{lemma}
The proof of Lemma \ref{lemma2} is a combination of two arguments of Tolsa: the proofs of \cite[Theorem 1.1]{To2} and \cite[Lemma 2.2]{To3}. I start with an analogue of \cite[Theorem 1.1]{To2}:

\begin{lemma}\label{lemma1} Assume that $\mu \in L^{2}(\nu)$. Then
\begin{displaymath} \mathop{\sum_{I \in \calD}}_{\nu(I) > 0} \alpha_{\mu,\nu}^{2}(I)\frac{\mu(I)^{2}}{\nu(I)} \lesssim \|\mu\|_{L^{2}(\nu)}^{2}. \end{displaymath}
\end{lemma}

\begin{proof} It suffices to sum over the intervals $I \subset \calD$ with $\mu(I) > 0$ and $\nu(I) > 0$; fix one of these $I$, and a $1$-Lipschitz function $\psi \colon \R \to \R$, supported on $[0,1]$. Then, write
\begin{equation}\label{form35} \left| \int \psi \, d\mu_{I} - \int \psi \, d\nu_{I} \right| = \left| \frac{1}{\mu(I)} \int_{I} (\psi \circ T_{I})g  d\nu - \frac{1}{\nu(I)} \int_{I} (\psi \circ T_{I}) \, d\nu\right|,\end{equation}
where $g$ is the Radon-Nikodym derivative $d\mu/d\nu \in L^{2}(\nu)$. Express $g\chi_{I}$ in terms of standard ($\nu$-adapted) martingale differences:
\begin{equation}\label{form41} g \chi_{I} = \langle g \rangle^{\nu}_{I}\chi_{I} + \sum_{J \in \calD(I)} \Delta_{J}^{\nu} g, \end{equation}
where $\calD(I) := \{J \in \calD : J \subset I\}$, the sum converges in $L^{2}(\nu)$, and
\begin{displaymath} \langle g \rangle^{\nu}_{I} = \frac{1}{\nu(I)} \int g \, d\nu = \frac{\mu(I)}{\nu(I)} \quad \text{and} \quad \Delta_{J}^{\nu}g = -\langle g \rangle_{J}^{\nu}\chi_{J} + \sum_{J' \in \ch(J)} \langle g \rangle_{J'}^{\nu}\chi_{J'}. \end{displaymath}
Note that $\Delta_{J}^{\nu}g$ is supported on $J$ and has $\nu$-mean zero. By \eqref{form41},
\begin{equation}\label{form34} \frac{1}{\mu(I)} \int_{J} (\psi \circ T_{I})g \, d\nu = \frac{1}{\nu(I)} \int_{I} (\psi \circ T_{I}) \, d\nu + \sum_{J \in \calD(I)} \frac{1}{\mu(I)}\int_{J} (\psi \circ T_{I})\Delta_{J}^{\nu}g \, d\nu. \end{equation}
Since the first term on the right hand side of \eqref{form34} cancels out the last term in \eqref{form35}, one can continue as follows:
\begin{align*} \eqref{form35} & \leq \sum_{J \in \calD(I)} \frac{1}{\mu(I)}\left| \int_{J} (\psi \circ T_{I})\Delta_{J}^{\nu} g \, d\nu \right|\\
& = \sum_{J \in \calD(I)} \frac{1}{\mu(I)}\left| \int_{J} [(\psi \circ T_{I}) - (\psi \circ T_{I}(x_{J}))] \Delta_{J}^{\nu} g \, d\nu \right|. \end{align*}
Above, $x_{J}$ is the midpoint of $J$, and the mean zero property of $\Delta_{J}^{\nu}g$ was used. Finally, recalling that $\psi$ is $1$-Lipschitz, one obtains
\begin{displaymath} \eqref{form35} \leq \sum_{J \in \calD(I)} \frac{\ell(T_{I}(J))}{\mu(I)} \|\Delta_{J}^{\nu}g\|_{L^{1}(\nu)} \leq \sum_{J \in \calD(I)} \frac{\ell(J)\nu(J)^{1/2}}{\mu(I)\ell(I)}\|\Delta_{J}^{\nu} g\|_{L^{2}(\nu)}. \end{displaymath}
Taking a $\sup$ over admissible functions $\psi \colon \R \to \R$ gives
\begin{equation}\label{form36} \alpha_{\mu,\nu}(I) \leq \sum_{J \in \calD(I)} \frac{\ell(J)\nu(J)^{1/2}}{\mu(I)\ell(I)}\|\Delta_{J}^{\nu}g\|_{L^{2}(\nu)}. \end{equation}
Now, using \eqref{form36} and Cauchy-Schwarz, we may sum over $I \in \calD$ as follows (we suppress the requirement $\nu(I) > 0$ from the notation):
\begin{align*} \sum_{I \in \calD} \alpha^{2}_{\mu,\nu}(I)\frac{\mu(I)^{2}}{\nu(I)} & \leq \sum_{J \in \calD} \left(\sum_{J \in \calD(I)} \frac{\ell(J)\nu(J)^{1/2}}{\ell(I)} \|\Delta_{J}^{\nu} g\|_{L^{2}(\nu)} \right)^{2}\frac{1}{\nu(I)}\\
& \leq \sum_{I \in \calD} \left(\sum_{J \in \calD(I)} \frac{\ell(J)}{\ell(I)} \|\Delta_{J}^{\nu} g\|_{L^{2}(\nu)}^{2} \right) \sum_{J \in \calD(I)} \frac{\ell(J)\nu(J)}{\ell(I)\nu(I)}. \end{align*} 
Clearly,
\begin{displaymath}  \sum_{J \in \calD(I)} \frac{\ell(J)\nu(J)}{\ell(I)\nu(I)} \lesssim 1, \end{displaymath}
so
\begin{displaymath} \sum_{J \in \calD} \alpha_{\mu,\nu}(I)^{2}\frac{\mu(I)^{2}}{\nu(I)} \lesssim \sum_{J \in \calD} \|\Delta_{J}^{\nu}g\|_{L^{2}(\nu)}^{2} \sum_{I \supset J} \frac{\ell(J)}{\ell(I)} \lesssim \sum_{J \in \calD} \|\Delta_{J}^{\nu}g\|_{L^{2}(\nu)}^{2} \leq \|g\|_{L^{2}(\nu)}^{2}, \end{displaymath}
as claimed. \end{proof}

\begin{cor}\label{cor1} If $\mu \in L^{2}(\nu)$, then $\calS_{\calD,\nu}(\mu)$ is finite $\mu$ almost everywhere. \end{cor}

\begin{proof} By Lemma \ref{lemma1}, and the Lebesgue differentiation theorem, the following conditions hold $\mu$ almost everywhere:
\begin{displaymath} \sum_{x \in I \in \calD} \alpha_{\mu,\nu}(I)^{2}\frac{\mu(I)}{\nu(I)} < \infty \quad \text{and} \quad \exists \lim_{I \to x} \frac{\mu(I)}{\nu(I)} = \mu(x) > 0.  \end{displaymath}
Clearly $\calS_{\calD,\nu}(\mu)(x) < \infty$ for such $x \in [0,1)$. \end{proof}

Now, we can prove Lemma \ref{lemma2} by an argument similar to \cite[Lemma 2.2]{To3}:

\begin{proof}[Proof of Lemma \ref{lemma2}] Perform a Calder\'on-Zygmund decomposition of $\mu$ with respect to $\nu$, at some level $\lambda \geq 1$. More precisely, let $\calB$ be the family of maximal intervals $I \in \calD$ with $\mu(I) > \lambda \nu(I)$, and set $\mu = g + b$, where 
\begin{displaymath} g = \mu|_{G} + \sum_{I \in \calB} \frac{\mu(I)}{\nu(I)}\nu|_{I}, \qquad G := [0,1) \setminus \bigcup_{I \in \calB} I, \end{displaymath}
and
\begin{displaymath} b = \sum_{I \in \calB} \left[\mu|_{I} - \frac{\mu(I)}{\nu(I)}\nu|_{I} \right] =: \sum_{I \in \calB} b_{I}. \end{displaymath}
Then $\|g\|_{L^{\infty}(\nu)} \lesssim \lambda$ (the implicit constants depend on the doubling of $\nu$), and
\begin{displaymath} \nu([0,1) \setminus G) = \sum_{I \in \calB} \nu(I) < \frac{1}{\lambda} \sum_{I \in \calB} \mu(I) \leq \frac{1}{\lambda}. \end{displaymath}
Since $\mu_{a} \in L^{1}(\nu)$ (recall that $\mu$ is a finite measure), it follows that $\mu_{a}([0,1) \setminus G) \to 0$ as $\lambda \to \infty$. Hence, it suffices to show that
\begin{displaymath} \calS_{\calD,\nu}(\mu)(x) < \infty \text{ for $\mu$ almost every } x \in G \cap \spt_{\calD} \mu, \end{displaymath}
where $\spt_{\calD} \mu = \{x \in \R : \mu(I) > 0 \text{ for all } x \in I \in \calD\}$. Let $\calG \subset \calD$ be the intervals, which are not contained in any interval in $\calB$. Fix $x \in G \cap \spt_{\calD} \mu$, and note that if $x \in I \in \calD$, then $I \in \calG$. Observe that $\mu(I) = g(I)$ for $I \in \calG$, and consequently
\begin{align*} \left| \int \psi \, d\mu_{I} - \int \psi \, d\nu_{I} \right| & \leq \left| \int \psi \, d\mu_{I} - \int \psi \, dg_{I} \right| + \alpha_{g,\nu}(I)\\
& = \frac{1}{\mu(I)} \left| \int_{I} (\psi \circ T_{I}) \, db \right| + \alpha_{g,\nu}(I), \quad I \ni x,  \end{align*}
for any $1$-Lipschitz function $\psi \colon \R \to \R$ supported on $[0,1]$. Using the zero-mean property of the measures $b_{J}$, estimate further as follows:
\begin{displaymath} \left| \int (\psi \circ T_{I}) \, db \right| \leq \sum_{J \in \calB(I)} \left| \int (\psi \circ T_{I}) \, db_{J} \right| = \sum_{J \in \calB(I)} \left| \int [(\psi \circ T_{I}) - (\psi \circ T_{I}(x_{J})] \, db_{J} \right|, \end{displaymath}
where $\calB(I) := \{J \in \calB : J \subset I\}$, and $x_{J}$ is the midpoint of $J$. Using the fact that $\psi$ is $1$-Lipschitz, one has
\begin{displaymath} \frac{1}{\mu(I)} \left| \int_{I} (\psi \circ T_{I}) \, db \right| \leq \frac{1}{\mu(I)} \left| \int [(\psi \circ T_{I}) - (\psi \circ T_{I}(x_{J})] \, db_{J} \right| \leq \frac{\ell(T_{I}(J))}{\mu(I)} \|b_{J}\| \lesssim \frac{\ell(J)\mu(J)}{\ell(I)\mu(I)}, \end{displaymath}
and finally
\begin{displaymath} \calS_{\calD,\nu}^{2}(\mu)(x) \lesssim \calS_{\calD,\nu}(g)^{2}(x) + \sum_{x \in I \in \calG} \left(\sum_{J \in \calB(I)} \frac{\ell(J)\mu(J)}{\ell(I)\mu(I)} \right)^{2} =: \calS_{\calD,\nu}(g)^{2}(x) + S^{2}(x). \end{displaymath}
Since $\calS_{\calD,\nu}(g)$ is finite $g$ almost everywhere by Corollary \ref{cor1}, and in particular $\calS_{\calD,\nu}(g)(x) < \infty$ for $\mu$ almost every $x \in G$, it remains to prove that $S(x) < \infty$ for $\mu$ almost every $x \in \R$. First, note that
\begin{displaymath} \sum_{J \in \calB(I)} \frac{\ell(J)\mu(J)}{\ell(I)\mu(I)} \leq \frac{1}{\mu(I)} \sum_{J \in \calB(I)} \mu(J) \leq 1, \end{displaymath}
as the intervals in $\calB(I)$ are disjoint. Consequently,
\begin{align*} \int S^{2} \, d\mu & \leq \int \sum_{x \in I \in \calG} \sum_{J \in \calB(I)} \frac{\ell(J)\mu(J)}{\ell(I)\mu(I)} \, d\mu(x) = \sum_{I \in \calG} \sum_{J \in \calB(I)} \frac{\ell(J)\mu(J)}{\ell(I)}\\
& = \sum_{J \in \calB} \mu(J) \sum_{J \subset I \in \calG} \frac{\ell(J)}{\ell(I)} \lesssim \sum_{J \in \calB} \mu(J) \leq \|\mu\| < \infty. \end{align*}  
It follows that $S^{2}(x) < \infty$ for $\mu$ almost every $x \in \R$. This completes the proof of Lemma \ref{lemma2}, and Theorem \ref{main}(a). \end{proof}

\subsection{Bounding the non-dyadic square function} It remains to prove Theorem \ref{mainCont}(a). Assume that $\mu,\nu$ are Radon measures on $\R$, with $\nu$ doubling, and recall that $\calS_{\nu}(\mu)$ is the square function
\begin{displaymath} \calS_{\nu}^{2}(\mu)(x) = \int_{0}^{1} \alpha_{s,\mu,\nu}^{2}(B(x,r)) \, \frac{dr}{r}, \qquad x \in \R. \end{displaymath}
The claim is that $\calS_{\nu}(\mu)$ is finite $\mu_{a}$ almost everywhere; since this is a local problem, one may assume that $\mu$ is a finite measure. Now, as in Remark \ref{AinftyRemark} (or see \cite[Section 5]{MTT}), pick a finite number of dyadic systems $\calD_{1},\ldots,\calD_{N}$ with the following property: for any interval $I \subset \R$, there exists $j \in \{1,\ldots,N\}$, depending on $I$, and an interval $J \in \calD_{j}$ such that $I \subset J_{i}$ and $|J_{i}| \sim |I|$. As a little technical point, we actually need to restrict $\calD_{j}$ to intervals of length at most one, so also the defining property above only holds for intervals $I \subset \R$ of length $|I| \leq r_{0}$, say.

Then, apply Lemma \ref{lemma2} to each of the corresponding square functions $\calS_{\calD_{j},\nu}(\mu)$ to infer the following: 
\begin{displaymath} \calS_{\calD,\nu}(\mu)(x) := \sum_{j = 1}^{N} \calS_{\calD_{j},\nu}(\mu)(x) < \infty \end{displaymath}
for $\mu_{a}$ almost every $x \in \R$ (note that $\nu$ is dyadically doubling relative to every $\calD_{j}$). So, it suffices to argue that $\calS_{\calD,\nu}(\mu)$ dominates $\calS_{\nu}(\mu)$. Using the stability of the smooth $\alpha$-numbers, and the fact that they are dominated by the regular $\alpha$-numbers whenever $\nu$ is doubling (see Proposition \ref{basicProperties}), one has
\begin{displaymath} \alpha^{2}_{s,\mu,\nu}(B(x,r)) \lesssim \alpha_{\mu,\nu}^{2}(I^{j}_{x,r}), \qquad x \in \R, \: 0 < r < r_{0}, \end{displaymath}
where $j \in \{1,\ldots,N\}$, and $I_{x,r}^{j} \in \calD_{j}$ is a dyadic interval of length at most one, satisfying $x \in B(x,r) \subset I_{x,r}$ and $|I_{x,r}| \sim r$. The existence follows from the construction of the systems $\calD_{j}$. It is now clear that $\calS_{\nu}(\mu) \lesssim \calS_{\calD,\nu}(\mu)$, and the proof of Theorem \ref{mainCont}(a) is complete.  

\begin{remark} Lemma 5.4 in \cite{ADT2} implies that 
\begin{displaymath} \int_{1/4}^{1/2} \alpha_{\mu,\nu}(B(0,t)) \, dt \lesssim \alpha_{s,\mu,\nu}(B(0,1)), \end{displaymath}
whenever $\nu$ is doubling, and $\nu(B(0,1/4)) > 0$, $\mu(B(0,1/4)) > 0$. So, at the level of $L^{1}$-averages over scales, the smooth and regular $\alpha$-numbers are comparable. One would need a similar comparison at the level of $L^{2}$-averages to answer Question \ref{Q1}. \end{remark}

\end{document}